\documentclass[a4paper,11pt,leqno]{article}

\usepackage[latin1]{inputenc}
\usepackage[T1]{fontenc} 
\usepackage[english]{babel}
\usepackage{verbatim}
\usepackage{calligra}
\usepackage{enumerate}
\usepackage{dsfont}
\usepackage{amsfonts}
\usepackage{amsmath}
\usepackage{amsthm}
\usepackage{amssymb}
\usepackage{mathtools}
\usepackage{mathrsfs}
\usepackage{color}
\usepackage{graphicx}
\usepackage{bm}
  \usepackage[all]{xy}
\usepackage{hyperref}
\hypersetup{
    colorlinks=true,
    linkcolor=black,
    filecolor=black,      
    urlcolor=black,
}

\usepackage{tikz}

\usepackage{graphicx}		
\usepackage[hypcap=false]{caption}

\DeclareMathOperator*{\tend}{\longrightarrow}

\DeclareMathOperator*{\D}{\rm{div}}

\DeclareMathOperator*{\limss}{\overline{\rm lim}}

\theoremstyle{definition}
\newtheorem{defi}{Definition}

\newtheorem{rmk}[defi]{Remark}

\theoremstyle{plain}
\newtheorem{thm}[defi]{Theorem}
\newtheorem{prop}[defi]{Proposition}
\newtheorem{cor}[defi]{Corollary}
\newtheorem{lemma}[defi]{Lemma}

\newcommand{\mc}{\mathcal}

\newcommand{\what}{\widehat}

\newcommand{\R}{\mathbb{R}}
\newcommand{\Q}{\mathbb{Q}}

\newcommand{\Z}{\mathbb{Z}}

\renewcommand{\P}{\mathbb{P}}

\newcommand{\curl}{{\rm curl}\,}

\newcommand{\dx}{ \, {\rm d} x}
\newcommand{\dt}{ \, {\rm d} t}

\newcommand{\al}{\alpha}
\newcommand{\bt}{\beta}

\allowdisplaybreaks

\begin{document}

\title{\textsc{\Large{\textbf{Bounded Solutions in Incompressible Hydrodynamics}}}}

\author{\normalsize\textsl{Dimitri Cobb}\vspace{.5cm} \\
\footnotesize{\textsc{Universität Bonn}} \\
{\footnotesize \it Mathematisches Institut} \vspace{.1cm} \\
{\footnotesize Endenicher Allee 60, 53115 Bonn, Germany} \vspace{.1cm} \\
\footnotesize{\ttfamily{cobb@math.uni-bonn.de}}
}

\date\today

\maketitle

\subsubsection*{Abstract}
{\footnotesize In this article, we study bounded solutions of Euler-type equations on $\R^d$ which have no integrability at $|x| \rightarrow +\infty$. As has been previously noted, such solutions fail to achieve uniqueness in an initial value problem, even under strong smoothness conditions. This contrasts with well-posedness results that have been obtained by using the Leray projection operator in these equations. This apparent paradox is solved by noting that using the Leray projector requires an extra condition the solutions must fulfill at $|x| \rightarrow + \infty$. Our goal is to find one such condition which is sharp. We then apply the methods we develop to prove a full uniqueness result for Besov-Lipschitz solutions, as to  the theory of Serfati solutions. In the last Section, we see how these techniques also apply to the Elsässer variables used in ideal MHD.
}

\paragraph*{\small 2010 Mathematics Subject Classification:}{\footnotesize 35Q35 
(primary);
35A02, 
35B30, 
35Q31, 
35S30  
(secondary).}

\paragraph*{\small Keywords: }{\footnotesize Leray projection, incompressible fluids, bounded solutions, Euler equations, far field condition.}

{\footnotesize
\tableofcontents
}

\section{Introduction}

In this article, we study some specific features of incompressible fluids evolving in unbounded domains. In order to illustrate the methods we develop, we will focus mainly on the Euler equations, which concentrates the difficulties of the problem. However, we will also show how the principles we present can also be applied to a broader class of equations by applying them to the study of Elsässer variables in ideal magnetohydrodynamics.

\medskip

We are concerned with \textsl{bounded} solutions to systems of PDEs describing incompressible fluid mechanics. The simplest of these equations is the incompressible (homogeneous) Euler system, which reads
\begin{equation}\label{ieq:c2Euler}
\begin{cases}
\partial_t u + \D (u \otimes u) + \nabla \pi = 0\\
\D (u) = 0.
\end{cases}
\end{equation}
This system is set on the whole space $\R^d$, and its unknowns are the velocity field $u(t, x) \in \R^d$ and the pressure field $\pi(t, x) \in \R$. As we will see, the issues raised in this article are very specific to the non-compact nature of $\R^d$, so we will not be interested by flows defined on the torus $\mathbb{T}^d$.

\medskip

Our concern is the resolution of the Cauchy problem for \eqref{ieq:c2Euler}. Most PDEs require a form of boundary condition to be well-posed, and the Euler system is no different. But when working on the whole space $\R^d$, which has no boundary, one must rather require conditions at infinity in the form of a far field condition on the velocity: for example, we might demand the fluid to be at rest at infinity in some inertial reference frame given by a fixed vector $V \in \R^d$,
\begin{equation*}
u(t, x) \tend V \qquad \text{as } |x| \rightarrow + \infty.
\end{equation*}
The absolute necessity of such a condition is made obvious by a simple example: consider $f \in C^\infty (\R ; \R^d)$ and define a flow by
\begin{equation}\label{ieq:c2Unif}
u(t, x) = f(t) \qquad \text{ and } \qquad \pi(t, x) = - f'(t) \cdot x.
\end{equation}
Then the couple $(u, \pi)$ is a solution of \eqref{ieq:c2Euler}. But taking $f$ to be any function that is compactly supported away from $t = 0$, we construct an infinity of solutions that correspond to the initial data $u(0) = 0$. And this lack of uniqueness cannot be dismissed as an artifact of the Galilean invariance of the equations: there is no inertial reference frame in which the fluid \eqref{ieq:c2Unif} is always at rest at infinity.

\medskip

Note that this issue is of a different nature than the $C^m$ or $C^1\cap L^2$ ill-posedness displayed in \cite{BL} and \cite{EM}. For instance as the authors of \cite{EM} rely on singular integral operators not mapping $C^0 \cap L^2$ to $L^\infty$, while our example \eqref{ieq:c2Unif} hinges on a lack of far field-conditions.

\medskip

The questions we ask in this chapter are the following: \textsl{What type of far field condition should a (smooth) solution of} \eqref{ieq:c2Euler} \textsl{satisfy to be uniquely determined by the initial data? Can such a condition be ``optimal'', in the sense that any other condition has to be stronger?}

\medskip

Before diving into the bulk of this article, let us note that bounded solutions of infinite energy, such as \eqref{ieq:c2Unif}, are of specific interest in fluid mechanics: in particular, the space $W^{1, \infty}$ of Lipschitz functions is essentially\footnote{More precisely, it shoule be noted that Yudovich or Serfati solutions are \textsl{log-Lipschitz}, but this does not really undermine our point: it is not possible to work much below Lipschitz regularity.} the largest space in which one can hope to solve most equations of ideal fluid mechanics --the first equation in \eqref{ieq:c2Euler} is nearly a transport equation. So far, all well-posedness results have been proved for solutions in spaces strictly embedded in $W^{1, \infty}$ (or in the space $LL$ of $\log$-Lipschitz functions).

In addition, infinite energy solutions have also been at the center of much attention in the past twenty years, and present their own interesting challenges (see \cite{G} for an introduction to these questions), so that a precise theory of non-integrable solutions seems quite desirable.

\subsection{Role of the Pressure and Leray Projection}

The analysis of this problem hinges on our understanding of the pressure force $- \nabla \pi$. An intuitive reason $\pi$ has this key role can be seen from our counter-example. In \eqref{ieq:c2Unif}, the flow is subject to a forcing: a pressure differential coming from $|x| \rightarrow + \infty$ induces the motion of the fluid. And of course, any fluid that is under the influence of an arbitrary forcing cannot have its dynamics solely determined by initial data.

\medskip

Now, although the pressure is formally one of the unknowns of the system, it is entirely determined (up to a constant) by the velocity: ultimately, only the velocity field really matters when solving the system. In fact, we may formally compute $\nabla \pi$ as a function of $u$ by taking the divergence of the first equation in \eqref{ieq:c2Euler} and solving the elliptic equation thus produced: we obtain
\begin{equation}\label{ieq:c2PressureForce}
\nabla \pi = \nabla (- \Delta)^{-1} \sum_{j, k} \partial_j \partial_k (u_j u_k).
\end{equation}
From the previous equation, we introduce the Leray projection operator
\begin{equation*}
\P = {\rm Id} + \nabla (- \Delta)^{-1} \D
\end{equation*}
which may be seen as the $L^2$-orthogonal projector on the subspace of divergence-free functions. In other words, we may see the pressure term in $\nabla \pi$ in \eqref{ieq:c2Euler} as an orthogonal projection enforcing the divergence-free condition $\D(u) = 0$ at all times. Applying $\P$ to \eqref{ieq:c2Euler}, we get a new equation
\begin{equation}\label{ieq:c2PEuler2}
\partial_t u + \P \D (u \otimes u) = 0,
\end{equation}
which is (up to a commutator term) a transport equation.

\medskip

The importance of this new equation lies in that, unlike the original Euler system \eqref{ieq:c2Euler}, the projected equation \eqref{ieq:c2PEuler2} is \textsl{well-posed} in a space of (regular enough) bounded functions: for example, Pak and Park prove in 2004 \cite{PP} that for any Besov-Lipschitz initial datum $u_0 \in B^1_{\infty, 1}$, there is a unique (local) solution $u \in C^0([0, T[ ; B^1_{\infty, 1})$ for some $T > 0$. 

However, we see that systems \eqref{ieq:c2Euler} and \eqref{ieq:c2PEuler2} are not equivalent. For instance, the uniform flow $u(t, x) = f(t)$ of \eqref{ieq:c2Unif} is a solution of \eqref{ieq:c2Euler} but not of \eqref{ieq:c2PEuler2}. While this may seem a bit surprising, general unpleasantness must be expected when dealing simultaneously with constant functions, as in \eqref{ieq:c2Unif}, and Leray projection: constant functions are gradients and divergence-free functions at the same time.

The difference between \eqref{ieq:c2Euler} and  \eqref{ieq:c2PEuler2} is that the projected equations contain an implicit far field condition: the pressure force can only be expressed by Leray projection \eqref{ieq:c2PressureForce} under a (to be determined) far field condition. The obvious reason for this is the presence of the inverted Laplace operator $(- \Delta)^{-1}$, which has its range in a space of functions with no harmonic component. Because the pressure is given by a Poisson equation\footnote{From now on, we agree that there is an explicit summation on repeated indices. Here, for instance, there is a sum on $j, k = 1, ..., d$.}
\begin{equation*}
- \Delta \pi = \partial_j \partial_k (u_j u_k),
\end{equation*}
it is determined in $\mc S'$ up to the addition of a harmonic polynomial. For every $t$, there is a $Q(t) \in \R[X]$ such that
\begin{equation*}
\partial_t u + \P \D (u \otimes u) + \nabla Q = 0,
\end{equation*}
so that a solution $u$ of the original Euler system \eqref{ieq:c2Euler} solves the projected system \eqref{ieq:c2PEuler2} if and only if $\nabla Q = 0$.

\medskip

Regarding our question of whether a solution $u$ of \eqref{ieq:c2Euler} may be determined by the initial datum $u_0$ alone, we see it depends on \textsl{knowing exactly when a solution of the Euler equations is a solution of the projected system \eqref{ieq:c2PEuler2}}.

While this question, and the presence of the harmonic polynomial $Q$, is sometimes dismissed as being irrelevant, equation \eqref{ieq:c2PressureForce} giving the ``canonical'' choice of the pressure corresponding to a loosely stated far field condition, the question we ask (and solve) is much deeper, as we wish to know what property of the flow this choice is related to.

\begin{rmk}
It should be noted that our problem is due to the fact that the solutions $(u, \pi)$ are defined on the \textsl{non-compact} space $\R^d$. For the Euler problem set on the torus $\mathbb{T}^d$, the issue completely disappears. The pressure solves a Poisson equation, and hence is given by its Fourier coefficients
\begin{equation*}
\forall k \in \mathbb{Z}^d \backslash \{ 0 \}, \qquad \what{\pi}(k) = - \sum_{j, l} \frac{k_j k_l}{|k|^2} \what{u_j u_l} (k)
\end{equation*}
and a constant function (for $k = 0$), which is irrelevant as it does not change the value of the pressure force $- \nabla \pi$ (in other words, the mean value of the solution is preserved). Therefore, the velocity field does indeed solve the projected equation \eqref{ieq:c2PEuler2}. 

We point out that there is a slight difference between solutions defined on $\mathbb{T}^d$ and periodic flows on $\R^d$, as these last solutions, such as \eqref{ieq:c2Unif}, may be driven by an exterior pressure. In that case, the pressure force $- \nabla \pi$ is periodic, although the pressure itself is not.
\end{rmk}

\subsection{Previous Results}\label{ss:c2PreviousResults}

This question has been abundantly studied in the past twenty years, but mainly for the Navier-Stokes equations, although all existing results can be applied to the Euler system with suitable adaptations. We mention some of the notable advances on the topic. Throughout this paragraph, $(u, \pi)$ is a solution of the Euler system \eqref{ieq:c2Euler}.

\medskip

The first kind of result is a condition on the pressure $\pi$ for it to be given by \eqref{ieq:c2PressureForce}. For instance, if $\pi$ is of the form
\begin{equation*}
\pi = \pi_0 + \sum_{i, j} R_i R_j \pi_{ij},
\end{equation*}
where the $R_k = \partial_k (- \Delta)^{-1/2}$ are the Riesz transforms and $\pi_{ij}(t), \pi_0(t) \in L^\infty$, then \eqref{ieq:c2PressureForce} holds, as proven in 2000 by Giga, Inui, J. Kato and Matsui \cite{GIKM}. Later, J. Kato \cite{KatoJ} extended this result: in order for \eqref{ieq:c2PressureForce} to be true, one only needs $\pi(t) \in {\rm BMO}$. Other results include those of Kukavica and Vicol \cite{KV}, who require that $|\pi(t, x)| = o(|x|)$ as $|x| \rightarrow + \infty$, or Maremonti \cite{Maremonti} with a finite moment condition $(1 + |x|)^{d+1} \pi(t, x) \in L^1$. Finally, Nakai and Yoneda \cite{NY} introduce a condition based on Campanato-type spaces.

A second type of theorem shows that any solution $u$ of the Euler system \eqref{ieq:c2Euler} must be given by transformation of a solution $v$ of \eqref{ieq:c2PEuler2} through a ``generalized Galilean transform'':
\begin{equation}\label{ieq:c2GenGalInv}
u(t, x) = v\big(t, x - G(t) \big) + g(t) \qquad \text{where } G(t) = \int_0^t g(s) {\rm d}s.
\end{equation}
See for example the articles of Kukavica \cite{Kukavica} and Kukavica and Vicol \cite{KV}. The generalized Galilean transform places the fluid in an accelerated reference frame in which an inertial force appears. This force can be seen as a pressure differential $\nabla (g'(t) \cdot x)$, and interpreted as a forcing term.

Finally, a last type of result refers to the velocity field only: a far field condition based on $u(t, x)$ is sufficient to constrain the form of the pressure. For instance, in \cite{CF3}, we have proven that a loose integrability condition at $|x| \rightarrow +\infty$ is enough for \eqref{ieq:c2PressureForce} to be true. Similarly, the recent work of Fern\'andez-Dalgo and Lemarié-Rieusset \cite{FL}, which explores the problem with much detail, provides another sufficient condition based on a finite momentum condition: the flow is required to satisfy
\begin{equation}\label{ieq:c2FLcondition}
\forall \, T > 0, \qquad \int_0^T \int \frac{|u(t, x)|^2}{(1+|x|)^d} \dx \dt < + \infty.
\end{equation}
Finally, we point out a last result, which is given by Lemarié-Rieusset in \cite{LM} (point \textit{ii} of Theorem 11.1, pp. 109--111), where the flow is required to satisfy the Morrey-type condition
\begin{equation}\label{ieq:c2LMcondition}
\forall t_1 < t_2, \qquad \sup_{x \in \R^d} \frac{1}{\lambda^d} \int_{t_1}^{t_2} \int_{|x - y| \leq \lambda} \big|u(t, y) \big|^2 {\rm d} y \dt \tend_{\lambda \rightarrow + \infty} 0.
\end{equation}
These last conditions of \cite{CF3}, \cite{FL} and \cite{LM}, although they are very general, cannot be optimal (both sufficient and neccessary). There are solutions of the projected problem \eqref{ieq:c2PEuler2} which do not fulfill conditions \eqref{ieq:c2FLcondition}, \eqref{ieq:c2LMcondition} or the one in \cite{CF3}. For instance, these conditions fail for any smooth nonzero periodic solution of \eqref{ieq:c2PEuler2}, such as \textsl{e.g.} the constant flow $u(t, x) = Cst$.

\medskip

Finally, we refer to \cite{Kelliher} for various results concerning the uniqueness of Serfati flows (see \cite{Serfati}, \cite{AKLFNL}) as solutions of the Euler equations. In Subsection \ref{ss:c2Serfati}, we discuss some of these and their relation with our own result.

\subsection{Main Result}

We now present and discuss the main theorem of this chapter. In order to introduce the formal ideas behind our main result, Theorem \ref{it:ThEuler} below, consider a smooth bounded solution $u$ of the Euler system which has bounded derivatives. As we have explained above, the pressure field is a solution of a Poisson equation
\begin{equation*}
- \Delta \pi = \partial_j \partial_k (u_j u_k)
\end{equation*}
and the issue lies in the fact that the solutions of this elliptic problem are not unique, as they are given, in $\mc S'$, up to the addition of a harmonic polynomial. Therefore, in order to recover the pressure force $- \nabla \pi$, we must add to \eqref{ieq:c2PressureForce} the gradient of a harmonic polynomial $Q \in \R [X]$, so as to obtain
\begin{equation}\label{ieq:c2EulerWithPolynomial}
\partial_t u + \P \D (u \otimes u) + \nabla Q = 0.
\end{equation}
Therefore, $u$ will be a solution of the projected system \eqref{ieq:c2PEuler2} if and only if $\nabla Q \equiv 0$.

\medskip

A crucial element of our argument is that polynomial functions are spectrally supported at $\xi = 0$. This means that to determine whether $\nabla Q$ is nonzero or not, it is enough to study the \textsl{low frequency} behavior of \eqref{ieq:c2EulerWithPolynomial}. Consider a nonnegative and compacty supported cut-off function $\chi$ with $\chi(\xi) = 1$ around $\xi = 0$. Then, by applying the Fourier multiplication operator $\chi (\lambda D)$ to \eqref{ieq:c2EulerWithPolynomial} for $\lambda > 0$, we get\footnote{We adopt the notation $D = - i \nabla$. More generally, for any function $\phi : \R^d \tend \R^d$, we note $\phi(D)$ the Fourier multiplication operator defined by the Fourier transform $\mc F [\phi (D) f](\xi) := \phi(\xi) \what{f}(\xi)$ for all $f \in \mc S$.}
\begin{equation*}
\partial_t \chi(\lambda D) u + \chi(\lambda D) \P \D (u \otimes u) + \nabla Q = 0,
\end{equation*}
and by letting $\lambda \rightarrow + \infty$, we hope to recover information on $\nabla Q$ based on the properties of $u$. Focus for the moment on the convective term $\chi(\lambda D) \P \D (u \otimes u)$. The symbol of the operator $\P \D$ is a homogeneous function of order $1$, and so is $O(|\xi|)$ at low frequencies. It is therefore extremely tempting to resort to the first Bernstein inequality (Lemma \ref{l:c2bern}) and write that the low frequency limit fulfills
\begin{equation}\label{ieq:c2PDiv}
\big\| \chi(\lambda D) \P \D (u \otimes u) \big\|_{L^\infty} = O \left( \frac{1}{\lambda} \right) \qquad \text{as } \lambda \rightarrow + \infty.
\end{equation}
Unfortunately, the Bernstein inequalities do not apply to the symbol of $\P \D$ because it lacks regularity: its derivatives are not continuous (see the assumptions of Lemma \ref{l:c2bern} below). Nevertheless, we will see that it is still possible to show that \eqref{ieq:c2PDiv} is true, but by using more involved computations. The convergence \eqref{ieq:c2PDiv}, which implies that $\partial_t \chi(\lambda D) u \tend \nabla Q$, shows that $\nabla Q \equiv 0$ if and only if
\begin{equation}\label{ieq:SpHconvDu}
\chi (\lambda D) \partial_t u \tend 0 \qquad \text{as } \lambda \rightarrow + \infty
\end{equation}
at all times. This motivates the introduction of the space $\mc S'_h$ of distributions whose Fourier transforms are weakly vanishing at $\xi = 0$. This space was defined by Chemin in the mid-90s in the context of the study of Besov spaces.

\begin{defi}
Let $\chi \in \mc D$ be the cut-off function defined above. The space $\mc S'_h$ is the set of all tempered distributions $f \in \mc S'$ such that  we have
\begin{equation*}
\chi(\lambda \xi) \what{f}(\xi) \tend_{\lambda \rightarrow + \infty} 0 \qquad \text{in } \mc S'.
\end{equation*}
In particular, there is no nontrivial polynomial function in $\mc S'_h$, that is $\mc S'_h \cap \R[X] = \{ 0 \}$.
\end{defi}

The convergence \eqref{ieq:SpHconvDu} can therefore be noted $\partial_t u \in \mc S'_h$. We will prove the following equivalence theorem, whose statement is a natural outcome of the discussion above.

\begin{thm}\label{it:ThEuler}
Let $T > 0$, $u_0 \in L^\infty$ and $u \in C^0([0, T[; L^\infty)$ be a weak solution of the Euler equations \eqref{ieq:c2Euler} associated to the initial datum $u_0$ and to a pressure $\pi \in \mc D'([0, T[ \times \R^d)$. Then the following assertions are equivalent:
\begin{enumerate}[(i)]
\item the flow $u$ solves the projected problem with initial datum $u(0)$ that satisfies $u(0) - u_0 = {\rm Cst} \in \R$,
\item for all times $t \in [0, T[$, we have $u(t) - u(0) \in \mc S'_h$,
\item for all times $t \in [0, T[$, we have $u(t) - u(0) \in {\rm BMO}^{-1}$,
\item the pressure satisfies $\pi \in C^0(]0, T[ ; {\rm BMO})$,
\item the pressure force is continuous with respect to time $\nabla \pi \in C^0(]0, T[; \mc S')$ and $\nabla \pi(t) \in \mc S'_h$ for all $0 < t < T$.
\item the pressure defines a continuous function in the topology of locally integrable functions modulo constants $\pi \in C^0(]0, T[ ; L^1_{\rm loc}/ \, \R)$, and we have $\chi(D) \pi(t,x) = O \big( \log |x| \big)$ as $|x| \rightarrow + \infty$ for all $0 < t < T$.
\end{enumerate}
In the above, the space ${\rm BMO}^{-1}$ is the linear space of derivatives of ${\rm BMO}$ functions (see the definition in Section \ref{s:preliminaires} below).
\end{thm}

\begin{rmk}
The distinction made in assertion \textit{(i)} between the initial datum $u_0$ and the initial value $u(0)$ may seem unusual. The initial value problem for bounded solutions has a few difficulties: initial data are only defined up to an additive constant. However, these technicalities are not at the core of the proof.
\end{rmk}

\begin{rmk}
In assertion \textit{(ii)}, the condition $u(t) - u(0) \in \mc S'_h$ is simply an integrated version of the one $\partial_t u \in \mc S'_h$ we isolated above.
\end{rmk}

\begin{rmk}
A small clarification must be made concerning points \textit{(iv)} and \textit{(vi)} in Theorem \ref{it:ThEuler}. Because ${\rm BMO} \subset L^1_{\rm loc} / \, \R$ is a space of functions defined modulo constants (see Section \ref{s:preliminaires}), the condition $\pi \in C^0(]0, T[ ; {\rm BMO})$ does not mean that $\pi(t, x)$ is continuous with respect to time as a distribution, even in the $\mc D'$ topology. Besides, note that the condition $\chi(D) \pi(t,x) = O(\log |x|)$ in \textit{(vi)} is independent of the choice a particular element of the class of $\pi(t)$ modulo constants.
\end{rmk}

Let us compare Theorem \ref{it:ThEuler} with the other known results we mentioned. Firstly, unlike the ones we presented in Subsection \ref{ss:c2PreviousResults}, Theorem \ref{it:ThEuler} is \textsl{optimal} in the sense that we provide a \textsl{necessary and sufficient} condition for a bounded solution of the original Euler equations \eqref{ieq:c2Euler} to solve the projected equations \eqref{ieq:c2PEuler2}.

\medskip

Next, amongst the various results discussed in Subsection \ref{ss:c2PreviousResults}, we see that the condition $\pi(t) \in {\rm BMO}$ of J. Kato \cite{KatoJ} is optimal.\footnote{Here and below, we must point out remember that the results of \cite{KatoJ} and \cite{KV} are originally stated for the Navier-Stokes equations in different time regularities.} On the one hand, if $u \in L^\infty$ is a solution of the projected problem \eqref{ieq:c2PEuler2}, then $\pi$ is the image of a bounded function by a singular integral operator $\pi = (- \Delta)^{-1} \partial_j \partial_k (u_j u_k)$, and hence lies in ${\rm BMO}$. On the other hand, if $\pi(t) \in {\rm BMO}$, then we will show the pressure force satisfies $\nabla \pi \in \mc S'_h$, so $\nabla Q$ must be zero (see Proposition \ref{p:c1KockTataruSpH}).

\medskip

Concerning the condition $\pi(t, x) = o(|x|)$ of Kukavica and Vicol \cite{KV}, while it seems more general than statement \textit{(vi)} of Theorem \ref{it:ThEuler}, the following lemma shows that it is in fact stronger than assertion \textit{(v)}, and therefore implies \textit{(vi)}.

\begin{lemma}\label{l:c2KukavicaVicolCondition}
Let $f \in L^\infty_{\rm loc}$ be a function such that $f(x) = o(|x|)$ as $|x| \rightarrow + \infty$. Then $\nabla f \in \mc S'_h$.
\end{lemma}

\begin{proof}
First, write $f(x) = |x|\epsilon(x)$ with $\epsilon(x) \tend 0$ as $|x| \rightarrow + \infty$. In order to take advantage of this, we use the convolution product form of $\chi(\lambda D) \nabla f$. Fix $\psi_\lambda (x) = \lambda^{-d} \psi(\lambda^{-1}x)$ such that $\what{\psi_\lambda}(\xi) = \chi(\lambda \xi)$. We have
\begin{equation*}
\chi(\lambda D) \nabla f(x) = \nabla \psi_\lambda * f (x) = \frac{1}{\lambda} \int |y| \epsilon(y) \nabla \psi \left( \frac{x-y}{\lambda} \right) \frac{{\rm d}y}{\lambda^d}.
\end{equation*}
By changing variables in this integral, we obtain the upper bound
\begin{equation*}
| \nabla \psi_\lambda * f(t, x) | \leq \int |y| \epsilon(\lambda y) \left| \nabla \psi \left( \frac{x}{\lambda} - y \right) \right| {\rm d}y,
\end{equation*}
which tends to zero uniformly locally as $\lambda \rightarrow + \infty$ by dominated convergence. Therefore, we also have
\begin{equation*}
\chi (\lambda D) \nabla f \tend_{\lambda \rightarrow + \infty} 0 \qquad \text{in } \mc S'.
\end{equation*}
Concerning the second point of the Lemma, we must make sure that any $f \in {\rm BMO}$ satisfies $f(x) = o(|x|)$ as $|x| \rightarrow + \infty$.
\end{proof}

Finally, we comment on the last conditions \eqref{ieq:c2FLcondition} and \eqref{ieq:c2LMcondition}, which bear on the velocity field. It must be noted that the framework of \cite{FL} and \cite{LM} is, in a sense, more general because it deals with Navier-Stokes solutions that have locally finite energy, whereas only more regular Besov-Lipschitz solutions make sense for our ideal fluid equations. On the other hand, both \eqref{ieq:c2LMcondition} and \eqref{ieq:c2FLcondition} are particular cases of $\mc S'_h$ functions.

\begin{lemma}\label{l:c2LMconditions}
Consider $f \in L^\infty$. Assume that either one of the following conditions hold:
\begin{equation*}
\int \frac{|f(x)|^2}{(1 + |x|)^d} \dx < + \infty \qquad \text{or} \qquad \frac{1}{\lambda^d} \int_{|x-y| \leq \lambda} |f(y)|^2 {\rm d}y \tend_{\lambda \rightarrow + \infty} 0 \quad \text{in } L^\infty_{\rm loc}.
\end{equation*}
Then $f \in \mc S'_h$. In fact, the first condition implies the second one.
\end{lemma}

\begin{proof}
We start by showing that the first moment-type condition implies the second Morrey-type one. For all $\lambda > 0$, we have
\begin{equation*}
\frac{1}{\lambda^d} \int_{|x-y| \leq \lambda} |f(y)|^2 {\rm d}y = \int_{|x - y| \leq \lambda} \left( \frac{1+|y|}{\lambda} \right)^d \frac{|f(y)|^2}{(1+|y|)^d} \, {\rm d}y.
\end{equation*}
Set $g(y) = (1 + |y|)^{-d} |f(y)|^2$ and consider $r > 0$ and $\epsilon > 0$. Because $g \in L^1$ by assumption, we may fix a $R > 0$ such that
\begin{equation*}
\forall |x| \leq r, \qquad \int_{|x - y| \geq R} g(y)\, {\rm d}y \leq \epsilon.
\end{equation*}
Therefore, by separating in our integral the $y \in \R^d$ that are close to $x$ from those that are far away, we have, for all $|x| \leq r$ and $\lambda \geq R$,
\begin{equation*}
\begin{split}
\frac{1}{\lambda^d} \int_{|x-y| \leq \lambda} |f(y)|^2 {\rm d}y & = \int_{|x-y| \leq R} \left( \frac{1 + |y|}{\lambda} \right)^d g(y) \, {\rm d}y \; + \int_{R \leq |x-y| \leq \lambda} \left( \frac{1 + |y|}{\lambda} \right)^d g(y) \, {\rm d}y\\
& \leq \int_{|x-y| \leq R} \left( \frac{1 + |y|}{\lambda} \right)^d g(y) \, {\rm d}y \; + \epsilon \left( \frac{1 + r + \lambda}{\lambda} \right)^d \\
& \leq \left( \frac{1 + |x| + R}{\lambda} \right)^d \| g \|_{L^1} + \epsilon \left( 1 + \frac{1 + r}{\lambda} \right)^d.
\end{split}
\end{equation*}
By taking the limit superior of the quantity in the last line as $\lambda \rightarrow + \infty$, we see that for all $|x| \leq r$ and $\epsilon > 0$,
\begin{equation*}
\limss_{\lambda \rightarrow + \infty} \frac{1}{\lambda^d} \int_{|x-y| \leq \lambda} |f(y)|^2 {\rm d}y \leq \epsilon.
\end{equation*}
In other words, the integral above converges uniformly to zero on every ball $|x| \leq r$, so we have the desired $L^\infty_{\rm loc}$ convergence. Finally, we only have to show that the second condition implies $f \in \mc S'_h$. Consider $\psi_\lambda \in \mc S$ as in the proof of Lemma \ref{l:c2KukavicaVicolCondition} above and let $\epsilon > 0$. We fix a $R > 0$ such that $\| \mathds{1}_{|y| \geq R} \psi(y) \|_{L^1} \leq \epsilon$. We have
\begin{equation*}
| \chi (\lambda D)f(x) | \leq \int_{|x - y| \leq \lambda R} |\psi_\lambda (x-y) f(y) | \, {\rm d}y \; + \int_{|x - y| \geq \lambda R} |\psi_\lambda (x-y) f(y) | \, {\rm d}y.
\end{equation*}
Because $\psi \in \mc S$ is a Schwartz function and $f \in L^\infty$, the second integral is bounded by 
\begin{equation*}
\begin{split}
\int_{|x - y| \geq \lambda R} |\psi_\lambda (x-y) f(y) | \, {\rm d}y & \leq \| f \|_{L^\infty} \int_{|x-y| \geq \lambda R} |\psi_\lambda (x - y)| \, {\rm d}y\\
& = \| f \|_{L^\infty} \| \mathds{1}_{|y| \geq R} \psi(y) \|_{L^1} \leq \epsilon \| f \|_{L^\infty}.
\end{split}
\end{equation*}
On the other hand, we may apply the Cauchy-Schwarz inequality to the first integral, so as to obtain
\begin{equation*}
\begin{split}
\int_{|x - y| \leq \lambda R} |\psi_\lambda (x-y) f(y) | \, {\rm d}y & \leq \left( \int_{|x - y| \leq \lambda R} |f(y)|^2 \, {\rm d}y \right)^{1/2} \left( \int_{|x - y| \leq \lambda R} |\psi_\lambda (x-y)|^2 \, {\rm d}y \right)^{1/2} \\
& \leq \left( \frac{1}{\lambda^d} \int_{|x - y| \leq \lambda R} |f(y)|^2 \, {\rm d}y \right)^{1/2} \left( \int \left| \psi \left(\frac{x-y}{\lambda} \right) \right|^2 \, \frac{{\rm d}y}{\lambda^d} \right)^{1/2}\\
& = \| \psi \|_{L^2} \left( \frac{1}{\lambda^d} \int_{|x - y| \leq \lambda R} |f(y)|^2 \, {\rm d}y \right)^{1/2}.
\end{split}
\end{equation*}
Since, by assumption, this quantity converges uniformly locally to zero, it also does so in the $\mc S'$ topology.
\end{proof}

\begin{rmk}
In fact, the results of Lemma \ref{l:c2LMconditions} hold when $f$ is in the Morrey space of uniformly locally $L^2$ functions (see Definition 11.4 p. 108 in \cite{LM}) with pretty much the same proofs. These are the $f \in L^2_{\rm loc}$ such that the integrals
\begin{equation*}
\int_B |f(x)|^2 \dx \leq C
\end{equation*}
are bounded for all balls $B \subset \R^d$ by a constant that depends only on the volume of the balls $C = C(|B|)$. But we will not need such a level of generality.
\end{rmk}

\begin{rmk}
There are many functions in $\mc S'_h$ which do not fulfill the conditions of Lemma \ref{l:c2LMconditions}. For instance, any periodic function with average value zero is in $\mc S'_h$ without any of the properties of Lemma \ref{l:c2LMconditions} being true. Similarly, as shown in \cite{Cobb2}, the sign function $\sigma = \mathds{1}_{\R_+} - \mathds{1}_{\R_-}$ is $\mc S'_h$, while it is obvious that none of the properties of Lemma \ref{l:c2LMconditions} are fulfilled.
\end{rmk}

\subsection*{Acknowledgements}

I am extremely grateful to my advisor, Francesco Fanelli, for his continuous support to this work, and his patience with my getting somewhat sidetracked by this problem. Many thanks also go to my co-workers for their kindness and help. This work has been partially supported by the Deutsche Forschungsgemeinschaft (DFG, German Research Foundation) Project ID 211504053 - SFB 1060, and by the project CRISIS (ANR-20-CE40-0020-01), operated by the French National Research Agency (ANR).

\section{Preliminary Definitions}\label{s:preliminaires}

In this paragraph, we define some notation and function spaces, as well as provide some elements of harmonic analysis that are used throughout the paper. 

\subsection{The Homogeneous Littlewood-Paley Decomposition}

First of all, it will be convenient to define the \textsl{homogeneous Littlewood-Paley decomposition}. For this, we fix a smooth radial and compactly supported function $\chi$ that is supported in the ball $B(0, 2) \subset \R^d$, equal to $\chi(\xi) = 1$ for $|\xi| \leq 1$ and such that the function $r \mapsto \chi(re)$ is non-increasing over $\R_+$ for all $e \in \R^d$. Set $\varphi(\xi) = \chi(\xi) - \chi(2 \xi)$ and $\varphi_m(\xi) = \varphi(2^{-m}\xi)$ for all $m \in \Z$. Then the functions $(\varphi_m)_{m \in \Z}$ define a partition of unity over the frequency space
\begin{equation}\label{eq:partitionUnity}
\sum_{m \in \Z} \varphi_m (\xi) = 1 \qquad \text{for } \xi \neq 0.
\end{equation}
The sum above should be understood in the sense of pointwise convergence. This partition of unity allows us to define the operators $\Delta_m$ by Fourier multiplication
$$
\forall m \in \Z, \qquad \dot{\Delta}_m = \varphi_m(D),
$$
which are called the \textsl{Littlewood-Paley} operators. They are bounded in the $L^p \tend L^p$ for all $p \in [1, + \infty]$, and their norm depends solely on $\chi$. By virtue of the fact that the $\varphi_m$ form a partition of unity, it formally is possible to reconstruct any function $f$ by means of the $\dot{\Delta}_m$. This is the \textsl{homogeneous Littlewood-Paley decomposition}, or the formal identity
\begin{equation}\label{eq:LP}
\sum_{m \in \Z} \dot{\Delta}_m f = f.
\end{equation}
However, it must be noted that because the partition of unity \eqref{eq:partitionUnity} does not hold for $\xi = 0$, the homogeneous Littlewood-Paley decomposition cannot be true for all $f \in \mc S'$. In particular, it fails for polynomial functions (but not only! See \cite{Cobb2}), as their Fourier transform is supported at $\xi = 0$. The space $\mc S'_h$ which we have mentioned in the Introduction above is, broadly speaking, the space of distributions $f \in \mc S'$ for which the Littlewood-Paley decomposition \eqref{eq:LP} converges in $\mc S'$ topology.

\begin{defi}
The space $\mc S'_h$ is the set of all tempered distributions $f \in \mc S'$ such that we have
\begin{equation}\label{eq:definitionSph}
\chi(\lambda D) f \tend_{\lambda \rightarrow + \infty} 0 \qquad \text{in } \mc S'.
\end{equation}
\end{defi}

\begin{rmk}
There seems to be no consensus on the way the space $\mc S'_h$ should be defined: for example, the convergence \eqref{eq:definitionSph} is sometimes required to hold in the norm topology of $L^\infty$, as in \cite{BCD}. However this has no real bearing on this article: our results would remain true regardless. Let us mention the discussion in \cite{Cobb2} concerning $\mc S'_h$ and the different definitions that have been proposed.
\end{rmk}

The following topological property of $\mc S'_h$ will be used several times in the sequel.

\begin{prop}[see Proposition 19 in \cite{Cobb2}]\label{p:c2LinftyHClosed}
The space $L^\infty \cap \mc S'_h$ is closed in $L^\infty$ for the norm topology of $L^\infty$.
\end{prop}

\begin{proof}
Let $(f_n)_{n \geq 0}$ be a converging sequence of functions in $L^\infty \cap \mc S'_h$ whose limit is $f \in L^\infty$. We have, for all $\phi \in \mc S$,
\begin{align*}
\left| \left\langle \chi(\lambda D)f, \phi \right\rangle_{L^\infty \times L^1} \right| & \leq \left| \left\langle \chi (\lambda D)(f - f_n) , \phi \right\rangle_{L^\infty \times L^1} \right| +  \left| \left\langle \chi(\lambda D) f_n, \phi \right\rangle_{L^\infty \times L^1} \right| \\
& \leq \left\| \chi (\lambda D) (f - f_n) \right\|_{L^\infty} \| \phi \|_{L^1} + \left| \left\langle \chi(\lambda D) f_n, \phi \right\rangle_{\mc S' \times \mc S} \right|.
\end{align*}
The fact that the $\chi(\lambda D)f_n$ converge to $0$ in $\mc S'$ as $\lambda \rightarrow +\infty$ shows that we have, for all $n \geq 0$,
\begin{equation*}
\limss_{\lambda \rightarrow + \infty} \left| \left\langle \chi(\lambda D)f, \phi \right\rangle_{\mc S' \times \mc S} \right| \leq   C \left\| (f - f_n) \right\|_{L^\infty} \| \phi \|_{L^1}.
\end{equation*}
This term has limit $0$ as $n \rightarrow +\infty$ so that $f$ indeed lies in $L^\infty \cap \mc S'$.
\end{proof}

On of the main properties of the Littlewood-Paley decomposition is the way the functions $\dot{\Delta}_m f$ transform under the action of derivatives: because $\dot{\Delta}_m f$ is spectrally localized in an annulus of radius roughly $2^m$, applying $\nabla$ essentially amounts to a multiplication by $2^m$. This idea is the essence of the next Lemma, the Bernstein inequalities.

\begin{lemma}[Bernstein inequalities]\label{l:c2bern}
Let  $0<r<R$.   A constant $C$ exists so that, for any nonnegative integer $k$, any couple $(p,q)$ 
in $[1,+\infty]^2$, with  $p\leq q$,  and any function $u\in L^p$,  we  have, for all $\lambda>0$,
$$
\displaylines{
{\rm supp}\, ( \widehat u ) \subset   B(0,\lambda R)\quad
\Longrightarrow\quad
\|\nabla^k u\|_{L^q}\, \leq\,
 C^{k+1}\,\lambda^{k+d\left(\frac{1}{p}-\frac{1}{q}\right)}\,\|u\|_{L^p}\;;\cr
{\rm supp}\, ( \widehat u ) \subset \{\xi\in\R^d\,|\, r\lambda\leq|\xi|\leq R\lambda\}
\quad\Longrightarrow\quad C^{-k-1}\,\lambda^k\|u\|_{L^p}\,
\leq\,
\|\nabla^k u\|_{L^p}\,
\leq\,
C^{k+1} \, \lambda^k\|u\|_{L^p}\,.
}$$
\end{lemma}   

More generally, Fourier multipliers whose symbol are homogeneous functions of degree $s$ act on the functions $\dot{\Delta}_m f$ roughly as a multiplication by $2^{ms}$.

\begin{lemma}[see Lemma 2.2 in \cite{BCD}]\label{l:c2FourierMultiplier}
Let $\sigma(\xi)$ be a $C^k$ function away from $\xi = 0$ with $k = 2 \lfloor 1 + d/2 \rfloor$ and assume there is a degree $s \in \R$ such that $|\nabla^l \sigma (\xi)| \leq C |\xi|^{s - l}$ for all $\xi \neq 0$ and $0 \leq l \leq k$. Then there exists a constant depending only on $\sigma$ and on the the dyadic decomposition function $\chi$ such that, for all $p \in [1, +\infty]$,
\begin{equation*}
\forall m \in \mathbb{Z}, \forall f \in L^p, \qquad \| \dot{\Delta}_m \sigma(D) f \|_{L^p} \leq C 2^{ms} \| \dot{\Delta}_m f \|_{L^p}.
\end{equation*}
\end{lemma}

\subsection{Besov Spaces}

In this paragraph, we define the class of Besov spaces (both homogeneous and non-homogeneous). These are based on the Littlewood-paley decomposition.

\begin{defi}
We define the \emph{homogeneous Besov space} $\dot{B}^s_{p, r}$ as the set of those classes of distributions modulo polynomials $f \in \mc S' / \, \R[X]$ such that\footnote{In the sequel, we will note $[f] = \{ f + Q, \; Q \in \R[X] \}$ the class of $f \in \mc S'$ modulo polynomials. However, for the sake of conciseness, we will drop the brackets and note $f \in \mc S' / \, \R[X]$ when there is no ambiguity.}
\begin{equation*}
\| f \|_{\dot{B}^s_{p, r}} := \left\| \left( 2^{ms} \| \dot{\Delta}_m f \|_{L^p} \right)_{m \in \mathbb{Z}} \right\|_{\ell^r(\mathbb{Z})} < +\infty.
\end{equation*}
The space $\dot{B}^s_{p, r}$ is Banach for all values of $(s, p, r)$. Let us point out that $\dot{B}^s_{p, r}$ does not depend on the Littlewood-Paley decomposition function $\chi$ up to an isomorphism of Banach spaces (see Theorem 3.3 in \cite{YS}).
\end{defi}

Of course, it is highly unconvenient to work with classes of distributions modulo polynomial when dealing with non-linear PDEs, as is the case here. We therefore define realizations of some homogeneous Besov spaces as subspaces of $\mc S'$ (see \cite{Bourdaud} and \cite{BCD} for more on that topic).

\begin{defi}\label{d:c2BesovRealizSpH}
Consider $(s, p, r) \in \R \times [1, + \infty]^2$ such that we have
\begin{equation}\label{eq:subcrit}
s < \frac{d}{p} \qquad \text{or} \qquad s = \frac{d}{p} \text{ and } r=1.
\end{equation}
Then we define the space $\dot{\mathfrak{B}}^s_{p, r}$ as the set of all $f \in \mc S'_h$ such that $\| f \|_{\dot{B}^s_{p, r}} < + \infty$. The space $\mathfrak{B}^s_{p, r}$ is Banach for all values of $(s, p, r)$ such that \eqref{eq:subcrit} holds. Moreover, the linear map
\begin{equation*}
f \in \dot{B}^s_{p, r} \longmapsto \sum_{m \in \Z} \dot{\Delta}_m f \in \mathfrak{B}^s_{p, r}
\end{equation*}
defines an isometric isomorphism $\dot{\mathfrak{B}}^s_{p, r} \approx \dot{B}^s_{p, r}$. In particular, every $f \in \dot{B}^s_{p, r}$ is the image of an element of $\dot{\mathfrak{B}}^s_{p, r}$ by the natural projection map $\mc S' \tend \mc S' / \, \R[X]$.
\end{defi}

The main interest of homogeneous Besov spaces is that Fourier multipliers whose symbol are homogeneous functions act very naturally on them. For example, if $\P$ is the Leray projector defined in the Introduction above, then Lemma \ref{l:c2FourierMultiplier} shows that it is a bounded operator
$$
\P : \dot{B}^s_{p, r} \tend \dot{B}^s_{p, r}
$$
for all values of $(s, p, r) \in \R \times [1, + \infty]^2$. This fact is particularly useful for dealing with enpoint exponents $p \in \{ 1, + \infty \}$, where usual Calder\'on-Zygmund theory of Singular Integral Operators fails.

\medskip

We now focus on \textsl{non-homogeneous} Besov spaces. They are defined so as to avoid the question of (non)convergence of the Littlewood-Paley decomposition: define the operator $\Delta_{-1} = \chi(D)$, so that we have the identity
$$
\Delta_{-1} f + \sum_{m \geq 0} \dot{\Delta}_m f = f
$$
for \textsl{every} $f \in \mc S'$. This is the \textsl{non-homogeneous Littlewood-Paley decomposition}, which is always convergent in $\mc S'$. To simplify notation, we note $\Delta_m := \dot{\Delta}_m$ for $m \geq 0$.

\begin{defi}
Let $s\in\R$ and $1\leq p,r\leq+\infty$. The \textsl{non-homogeneous Besov space} $B^{s}_{p,r}\,=\,B^s_{p,r}(\R^d)$ is defined as the set of tempered distributions $u \in \mc S'$ for which
$$
\|f\|_{B^{s}_{p,r}}\,:=\,
\left\|\left(2^{ms}\,\|\Delta_m f \|_{L^p}\right)_{m \geq -1}\right\|_{\ell^r}\,<\,+\infty\,.
$$
The non-homogeneous space $B^s_{p, r}$ is Banach for all values of $(s, p, r) \in \R \times [1, + \infty]^2$.
\end{defi}

\subsection{Functions of Bounded Mean Oscillations}

Finally, we will need a few properties of the space ${\rm BMO}$ of functions of Bounded Mean Oscillations. 

\begin{defi}
The space ${\rm BMO}$ is the set of classes of locally integrable functions modulo constant functions $f \in L^1_{\rm loc} / \, \R$ such that\footnote{We will adopt the same notation as above for the class $[f] = \{ f + C, \; C \in \R \}$ of $f \in L^1_{\rm loc}$, or simply $f \in L^1_{\rm loc} / \, \R$ when there is no ambiguity.}
\begin{equation*}
\| f \|_{\rm BMO} := \sup_{B} \frac{1}{|B|} \int_B \big| f(x) - (f)_B \big| \dx < + \infty,
\end{equation*}
where the supremum ranges over all balls $B \subset \R^d$ and $(f)_B = |B|^{-1} \int_B f$ is the average value of $f$ on $B$. The space ${\rm BMO}$ is Banach.
\end{defi}

One of the main features of ${\rm BMO}$ is that Singular Integral Operators act continuously on it. In particular, the Leray projector is a bounded operator
$$
\P : {\rm BMO} \tend {\rm BMO},
$$
whose precise definition involves the Fefferman-Stein duality $(\mc H^1)' = {\rm BMO}$ with the Hardy space $\mc H^1$. We refer to \cite{Stein}, in particular Theorem 1 pp. 142--144, for more on this topic. 

\medskip

A useful result is the way ${\rm BMO}$ interacts with Besov spaces. The following embedding Lemma consists of the dual embeddings for the Hardy space $\dot{B}^0_{1, 1} \subset \mc H^1 \subset \dot{B}^0_{1, 2}$, which can be found in paragraph 2.2.2 p. 105 of \cite{Grafakos}.

\begin{lemma}\label{l:BMOembed}
The following embeddings hold: $\dot{B}^0_{\infty, 2} \subset {\rm BMO} \subset \dot{B}^0_{\infty, \infty}$.
\end{lemma}

\begin{rmk}\label{r:embeddBMO}
The embeddings of Lemma \ref{l:BMOembed} require an explanation. On the one hand, the inclusion map ${\rm BMO} \subset \dot{B}^0_{\infty, \infty}$ can only make sense if any $[f] \in {\rm BMO}$ has a representative $f \in \mc S'$. This follows for example from the following integral inequality:
\begin{equation}\label{eq:BMOinequ}
\int \frac{\big| f(x) - (f)_{B_1} \big|}{(1 + |x|)^{d+1}} \dx \leq C(d) \| f \|_{\rm BMO}
\end{equation}
where $B_1 \subset \R^d$ is the unit ball. A proof can be found in \cite{Stein}, see equation (2) in paragraph 1.1.4, pp. 141.\footnote{See also Remark 1.27 in \cite{CobbPhD} for a logarithmic version of the inequality.} Note in particular that ${\rm BMO}$ cannot contain any non-constant polynomial function.

For the inclusion $\dot{B}^0_{\infty, 2} \subset {\rm BMO}$, its meaning is that any $[f] \in \dot{B}^0_{\infty, 2}$ has a representative $f \in \mc S'$ such that $\| f \|_{\rm BMO} < + \infty$.
\end{rmk}

\medskip

We now define the space ${\rm BMO}^{-1}$ of derivatives of ${\rm BMO}$ functions, after the work \cite{KT} of Koch and Tataru, who introduced it as a set of possible initial daat for the Navier-Stokes equations.

\begin{defi}
Let ${\rm BMO}^{-1}$ be the set of $f \in \mc S'$ such that there are functions $g_1, ..., g_d \in {\rm BMO}$ with $f = \sum_k \partial_k g_k$.
\end{defi}

From Lemma \ref{l:BMOembed}, we gather that there is an embedding ${\rm BMO}^{-1} \subset \dot{B}^{-1}_{\infty, \infty}$. One of the advantages of ${\rm BMO}^{-1}$ is that its elements are distributions, and not classes of distributions modulo constants as in ${\rm BMO}$, as the expression $\sum_k \partial_k g_k$ above does not depend on the representatives of the $g_k$ modulo constants. In fact, we can go one step further.

\begin{prop}\label{p:c1KockTataruSpH}
We have the inclusion ${\rm BMO}^{-1} \subset \dot{\mathfrak{B}}^{-1}_{\infty, \infty}$. In particular, any function $f \in {\rm BMO}^{-1}$ is the sum of its homogeneous Littlewood-Paley decomposition, or in other words $f \in \mc S'_h$.
\end{prop}

\begin{proof}
We start by recalling that ${\rm BMO}^{-1}$ contains no nontrivial polynomial function. This follows from the fact that no polynomial of degree $1$ or higher can be in ${\rm BMO}$ (see inequality \eqref{eq:BMOinequ} in Remark \ref{r:embeddBMO}).

\medskip

Now, let $f \in {\rm BMO}^{-1}$. Because of the embedding ${\rm BMO}^{-1}\subset \dot{B}^{-1}_{\infty, \infty}$, and since $\dot{B}^{-1}_{\infty, \infty}$ can be realized as $\dot{\mathfrak{B}}^{-1}_{\infty, \infty}$, we see that $f$ can be written as a sum
$$
f = h + Q \in \dot{\mathfrak{B}}^{-1}_{\infty, \infty} \oplus \mc S'_h.
$$
As $Q = f-h \in {\rm BMO}^{-1} + \, \mc S'_h$ belongs in a space which contains no nontrivial polynomial function, we deduce that $Q = 0$, and so $f = h$. Therefore $f \in \dot{B}^{-1}_{\infty, \infty}$.
\end{proof}

\medskip

We also will need to evaluate the rough growth of ${\rm BMO}$. The John-Nirenberg inequality shows that it is essentially logarithmic, which is what the following Proposition endeavours to prove.

\begin{prop}
Let $f \in {\rm BMO}$. Then we have \label{p:logBMO}
$$
\Delta_{-1} f(x) - \Delta_{-1}f(0) = O \big( \log(1 + |x|) \big), \qquad \text{as } |x| \rightarrow + \infty.
$$
\end{prop}

\begin{proof}
A key element of the proof will be the use of the embedding ${\rm BMO} \subset \dot{B}^0_{\infty, \infty}$. For notational convenience, set $g = \Delta_{-1} f$. We wish to make use of the (homogeneous) Littlewood-Paley decomposition of $g$, but we must be careful as it may not converge. We therefore define the partial sum operator
$$
T_N = \sum_{-N}^0 \dot{\Delta}_m
$$
for $N \geq 0$. We then have
$$
|g(x) - g(0)| \leq | T_Ng(x) - T_Ng(0) | + \big| ({\rm Id} - T_N) g(x) - ({\rm Id} - T_N) g(0) \big|.
$$
In order to bound the first term of the righthand side, we note that
\begin{equation}\label{eq:BMOEQ1}
\begin{split}
| T_Ng(x) - T_Ng(0) | & \leq \sum_{-N}^0 \left| \dot{\Delta}_m g(x) - \dot{\Delta}_m g(0) \right| \\
& \leq 2 (N+1) \| g \|_{\dot{B}^0_{\infty, \infty}}.
\end{split}
\end{equation}
For the second term, we bound the difference by the mean value theorem:
$$
\big| ({\rm Id} - T_N) g(x) - ({\rm Id} - T_N) g(0) \big| \leq |x| \, \big\| \nabla ({\rm Id} - T_N) g \big\|_{L^\infty}.
$$
Let us write $\nabla ({\rm Id} - T_N)g$ as the sum of its Littlewood-Paley decomposition, by Proposition \ref{p:c1KockTataruSpH}. We deduce from the previous equation that
\begin{equation*}
\begin{split}
\big| ({\rm Id} - T_N) g(x) - ({\rm Id} - T_N) g(0) \big| & \leq |x| \sum_{m < -N} \big\| \nabla \dot{\Delta}_m g \big\|_{L^\infty} \\
& \leq |x| \, \| g \|_{\dot{B}^0_{\infty, \infty}} \sum_{m < -N} 2^m \\
& \leq C |x| \, \| g \|_{\dot{B}^0_{\infty, \infty}} 2^{-N}.
\end{split}
\end{equation*}
By taking $N$ such that $N = \log(1 + |x|)$, we may put the previous inequality together with \eqref{eq:BMOEQ1} in order to obtain the desired result:
\begin{equation}\label{eq:BMOEQ2}
| g(x) - g(0) | \leq C \| g \|_{\dot{B}^0_{\infty, \infty}} \log \big( 1 + |x| \big).
\end{equation}
\end{proof}

\begin{rmk}
In fact, we have proved a slightly more general result: if $f \in \mc S'$ is a representative of $[f] \in \dot{B}^0_{\infty, \infty}$ such that $f(x) = o(|x|)$ as $|x| \rightarrow + \infty$, then it follows that inequality \eqref{eq:BMOEQ2} also holds.
\end{rmk}

\section{Bounded Weak Solutions of the Euler System}

In this paragraph, we clarify the precise meaning of a bounded weak solution of the Euler system \eqref{ieq:c2Euler} and the projected problem \eqref{ieq:c2PEuler2}. This will require exploring further the properties of the operator $\P \D$.

\subsection{Weak Solutions of the Euler Equations}

We define and comment the notion of bounded weak solutions of the Euler equations: these read
\begin{equation}\label{eq:c2Euler}
\begin{cases}
\partial_tu + \D(u \otimes u) + \nabla \pi = 0\\
\D(u) = 0.
\end{cases}
\end{equation}

\begin{defi}\label{d:c2weakEuler}
Let $u_0 \in L^\infty$ be a divergence-free function and $T > 0$. A function $u \in L^2_{\rm loc}([0, T[; L^\infty)$ is said to be a \textsl{bounded weak solution} of the Euler problem \eqref{eq:c2Euler} associated to the initial datum $u_0$ if there is a pressure $\pi \in \mc D'([0, T[ \times \R^d)$ such that
\begin{equation}\label{eq:c2EulerWeak}
\begin{cases}
\partial_t u + \D(u \otimes u) + \nabla \pi = \delta_0(t) \otimes u_0(x) \\
\D(u) = 0
\end{cases}
\qquad \text{in } \mc D'([0, T[ \times \R^d).
\end{equation}
In the equation above, the tensor product $\delta_0(t) \otimes u_0(x)$ is defined by the relation
\begin{equation*}
\forall \phi \in \mc D ([0, T[ \times \R^d), \qquad \big\langle \delta_0 (t) \otimes u_0(x), \phi(t, x) \big\rangle := \int u_0(x) \cdot \phi(0, x) \dx.
\end{equation*}
\end{defi}

\begin{rmk}
Several remarks are in order concerning this definition. Firstly, by testing \eqref{eq:c2EulerWeak} against any divergence-free $\phi \in \mc D([0, T[ \times \R^d)$, we obtain the usual integral formulation 
\begin{equation}\label{eq:c2weakMomentum}
\int_0^T \int \Big\{ \partial_t \phi \cdot u + \nabla \phi : u \otimes u \Big\} \dx \dt + \int u_0 \cdot \phi(0) \dx = 0
\end{equation}
of the momentum equation. Strictly speaking, Definition \ref{d:c2weakEuler} is a bit stronger than the previous integral formulation because an extra assumption is made on the pressure: it has to define a distribution with respect to both \textsl{time and space}.
\end{rmk}

\begin{rmk}
Next, we note that, contrary to most evolution equations, there is no need for a weak solution $u$ to be continuous with respect to time, even in the $\mc D'$ topology: the presence of the pressure may compensate time singularities in the flow, as would be the case if, for example, $u(t, x) = \mathds{1}_{[1, + \infty[}(t) V$ where $V \in \R^d$ is a constant vector. As we will see, this problem arises only in the $L^\infty$ framework: $L^p$ regularity of the solutions (with $p < +\infty$) implies continuity with respect to time in some weak topology, since, in that case, we may indeed apply the Leray projection to obtain \eqref{eq:c2EulerP} below.
\end{rmk}

\begin{rmk}
Finally, we must comment on the notion of initial datum in Definition \ref{d:c2weakEuler}. Because the weak form of the Euler equations is taken in the sense of $\mc D' ([0, T[ \times \R^d)$, that is with $t = 0$ included in the time interval, any translation $u_0 + V$ (with $V \in \R^d$) of the initial datum is also a valid initial datum \textsl{for the same solution}:
\begin{equation*}
\partial_t u + \D(u \otimes u) + \nabla \big( \pi + \delta_0(t) \otimes V \big) = \delta_0(t) \otimes \big( u_0(x) + V \big).
\end{equation*}
Simply, a different choice of pressure allows for $u$ to be associated with infinitely many initial data, regardless of the time regularity of $u$ (it could be $C^1_T(L^\infty)$). Seen from the perspective of the integral form \eqref{eq:c2weakMomentum} of the momentum equation, the ambiguity of the initial data is a direct consequence of testing with \textsl{divergence-free} functions only: for any fixed $V \in \R^d$ and divergence-free $\phi \in \mc D([0, T[ \times \R^d)$,
\begin{equation*}
\int (u_0 + V) \cdot \phi(0) \dx = \int u_0 \cdot \phi(0) \dx .
\end{equation*}
These test functions are orthogonal, in the sense of the $\mc S' \times \mc S$ duality, to all gradients. Therefore, any (regular) flow $u$ which solves the Euler problem is also a weak solution (according to Definition \ref{d:c2weakEuler}) related to the initial data $u(0) + \nabla g$ (for any $g \in \mc S'$).

In the $L^p$ framework (with $p < +\infty$), this ambiguity totally disappears, as long as we require the initial data to be $L^p$. Indeed, any divergence-free $f \in L^p$ which is also the gradient of some tempered distribution $f = \nabla g$ must be a harmonic polynomial, since we would then have $\Delta g = 0$. Since there is no nonzero polynomial in $L^p$, this implies $f = 0$.

Our problem lies in the fact that the space $L^\infty$ intersects non-trivially with $\R [X]$, so that there are nonzero divergence-free functions which are also gradients: constant functions.
\end{rmk}

\medskip

As we said in the introduction, bounded solutions of the Euler equations \eqref{eq:c2Euler} are not unique, even under $C^\infty$ smoothness assumptions.

\subsection{Weak Solutions of the Projected Problem}

In this section, we focus on the projected problem: formally, one may apply the Leray projection operator to kill the pressure term and obtain an equation on $u$ only, namely
\begin{equation}\label{eq:c2EulerP}
\partial_t u + \P \D(u \otimes u) = 0.
\end{equation}
However, as we saw in Chapter 1, Fourier multiplication operators with bounded homogeneous symbols of degree zero, such as the Leray projection $\P$, are not properly defined on $L^\infty$. We must therefore clarify the meaning of equation \eqref{eq:c2EulerP} when the solutions are solely bounded, say $u \in L^2_T(L^\infty)$. 

\medskip

The general idea of this paragraph is to study the order one operator $\P \D$ instead of $\P$, which is a combination of operators the form $T_{j, k, l} := \partial_j \partial_k \partial_l (- \Delta)^{-1}$. By introducing the fundamental solution $E \in L^1_{\rm loc}$ of the Laplacian on $\R^d$ given by
\begin{equation*}
E(x) = \frac{C(d)}{|x|^{d-2}} \text{ if } d \geq 3 \qquad \text{ and } \qquad E(x) = C(2) \log |x| \text{ if } d = 2,
\end{equation*}
where the constant $C(d) > 0$ is such that $-\Delta E = \delta_0$, we may see $T_{j, k, l}$ as a convolution operator whose kernel is $\Gamma_{j, k, l} := \partial_j \partial_k \partial_l E$ (seen as a distribution). We prove the following Proposition, which is a decomposition of Lemarié-Rieusset \cite{LM} (see Lemma 11.1 p. 106) also used in the simpler setting of bounded solutions by Pak and Park \cite{PP} (see equations (20) and (21) p. 1161).

\begin{prop}\label{p:c2Gamma}
Consider $j, k, l \in \{1, ..., d\}$ and let $\Gamma_{j,k,l} = \partial_j \partial_k \partial_l E \in \mc D'$. Then $\Gamma_{j, k, l}$ is the sum of a compactly supported distribution and an $L^1$ function. In addition, for all $s \in \R$ and $p, r \in [1, + \infty]$, convolution by $\Gamma$ defines a bounded operator on the non-homogeneous Besov spaces
\begin{equation*}
T_{j, k, l} : B^s_{p, r} \tend B^{s-1}_{p, r}.
\end{equation*}
\end{prop}

\begin{rmk}
The Besov boundedness of $T_{j, k, l} : B^s_{p, r} \tend B^{s-1}_{p, r}$ is falsely obvious in the case $p = +\infty$. Although Lemma \ref{l:c2FourierMultiplier} makes it clear that the operator
\begin{equation*}
\tilde{T}_{j, k, l} := \sum_{m \in \Z} \dot{\Delta}_m \partial_j \partial_k \partial_l (- \Delta)^{-1}
\end{equation*}
is bounded in the $B^s_{\infty, r} \tend B^{s-1}_{\infty, r}$ topology, as the sum converges normally in $L^\infty$ when restricted to ranks $j \leq -1$, the two operators $\tilde{T}_{j, k, l}$ and $T_{j, k, l}$ might not coincide: they could differ by a polynomial. This is why we need two extra steps: firstly to define $T_{j, k, l}$, this is the purpose of Proposition \ref{p:c2Gamma}, and secondly to show that $\tilde{T}_{j, k, l}$ and $T_{j, k, l}$ are actually equal, this being the scope of Proposition \ref{p:OpIsHom} below.
\end{rmk}

\begin{proof}
By fixing a compactly supported function $\chi \in \mc D$ with $\chi(x) \equiv 1$ around $x=0$, we write $\Gamma_{j, k, l} = \chi \partial_j \partial_k \partial_l E + (1 - \chi) \partial_j \partial_k \partial_l E$. Because of the form of the fundamental solution $E$, its third derivatives are integrable at infinity:
\begin{equation}\label{eq:c2GammaEQ1}
\partial_j \partial_k \partial_l E (x) = O \left( \frac{1}{|x|^{d+1}} \right) \qquad \text{as } |x| \rightarrow + \infty.
\end{equation}
Therefore, the function $(1 - \chi) \partial_j \partial_k \partial_l E$ is $L^1$ and we have proved that $\Gamma_{j, k, l} \in \mc E' + L^1$, where $\mc E'$ is the space of compactly supported distributions. Concerning boundedness in Besov spaces, we note that the operator $\P \D$ is entirely defined for frequencies $\xi \neq 0$ by the Fourier transform 
\begin{equation*}
\forall \phi \in \mc D(\R^d \setminus \{ 0 \} ; \R^d \otimes \R^d), \qquad \mc F \big[ \P \D \mc F^{-1}[\phi] \big](\xi) = \sigma(\xi) \phi(\xi),
\end{equation*}
where $\sigma$ is the symbol of $\P \D$. Consequently, Lemma \ref{l:c2FourierMultiplier} makes it clear that for any $m \geq 0$ and $f \in B^s_{p, r}$ we have
\begin{equation*}
\| \Delta_m \P \D (f) \|_{L^p} \leq C 2^m \| \Delta_m f \|_{L^p}.
\end{equation*}
It only remains to estimate the low frequency block $\Delta_{-1} \P \D (f)$. Let $\psi \in \mc S$ such that the operator $\Delta_{-1}$ is the convolution by $\psi$, that is $\Delta_{-1}f = \psi * f$. We use the cut-off function $\theta$ as above and split the kernel into two parts
\begin{equation*}
\Delta_{-1} \P \D (f) = \psi * \Gamma_{j, k, l} * f = \psi * \partial_j \partial_k \partial_l (\theta E) +  \psi * \partial_j \partial_k \partial_l \big( (1 - \theta)E \big).
\end{equation*}
Integrating by parts in the convolution product yields
\begin{equation*}
\psi * \Gamma_{j, k, l} * f = \partial_j \partial_k \partial_l \psi * (\theta E) * f + \psi * \partial_j \partial_k \partial_l \big( (1 - \theta) E \big)*f.
\end{equation*}
Now, because $E$ is locally integrable and thanks to \eqref{eq:c2GammaEQ1}, both kernels
\begin{equation*}
\partial_j \partial_k \partial_l \psi * (\theta E) \in L^1 \qquad \text{and} \qquad \psi * \partial_j \partial_k \partial_l \big( (1 - \theta) E \big) \in L^1
\end{equation*}
are integrable functions. The Hausdorff convolution inequality guarantees that $\Delta_{-1} \P \D$ is a bounded $L^p \tend L^p$ operator for all $p \in [1, + \infty]$.
\end{proof}

Proposition \ref{p:c2Gamma} makes it possible to define the notion of bounded weak solution of the projected Euler system \eqref{eq:c2EulerP}.

\begin{defi}\label{d:c2EulerP}
Let $T > 0$ and $u_0 \in L^\infty$ be a bounded divergence free function. We say that $u \in L^2_{\rm loc}([0, T[ ; L^\infty)$ is a \textsl{bounded weak solution} of the projected Euler system \eqref{eq:c2EulerP} associated to the initial datum $u_0$ if
\begin{equation*}
\partial_t u + \P \D (u \otimes u) = \delta_0(t) \otimes u_0(x) \qquad \text{in } \mc D'([0, T[ \times \R^d).
\end{equation*}
\end{defi}

Bounded weak solutions of the projected system \eqref{eq:c2EulerP} behave in many ways better than weak solutions of the Euler system. For example, solutions are necessarily continuous with respect to time in the $B^{-1}_{\infty, \infty}$ topology. In addition, the pressure, which is given by 
\begin{equation}\label{eq:c2EulerPEQ1}
\pi = (- \Delta)^{-1} \partial_j \partial_k (u_j u_k) \qquad (\text{modulo constants})
\end{equation}
has minimal regularity properties: being the image of the bounded function $u_j u_k \in C^0_T(L^\infty)$ by the singular integral operator $(- \Delta)^{-1} \partial_j \partial_k$, it defines a ${\rm BMO}$ function
\begin{equation*}
\pi \in C^0([0, T[; {\rm BMO}).
\end{equation*}
However, it must be noted that because ${\rm BMO}$ is a space of functions defined up to a constant (and hence ${\rm BMO} \hookrightarrow \mc S' / \, \R[X]$), the pressure does not define a distribution on $[0, T[ \times \R^d$, as the constant summand in \eqref{eq:c2EulerPEQ1} may be extremely singular in the time variable. On the other hand, because the pressure force is the derivative of \eqref{eq:c2EulerPEQ1}, it unambiguously defines a ${\rm BMO}^{-1} \subset \mc S'$ function
\begin{equation*}
\nabla \pi \in C^0([0, T[; {\rm BMO}^{-1}) \subset \mc D'([0, T[ \times \R^d).
\end{equation*}

\section{Proof the Main Result}

We are now ready to prove our main result, Theorem \ref{it:ThEuler}. We will start by recalling the statement of Theorem \ref{it:ThEuler} and making a few additional comments. Then, the proof will start with the simpler case where the solutions are $C^1$ with respect to the time variable, before dealing with the full statement of Theorem \ref{it:ThEuler}.

\subsection{Statement of the Theorem and Comments}

For the readers convenience, we recall the statement of our main result.

\begin{thm}\label{t:c2Euler}
Let $T > 0$, $u_0 \in L^\infty$ and $u \in C^0([0, T[; L^\infty)$ be a weak solution of the Euler equations \eqref{eq:c2Euler} associated to the initial datum $u_0$ and to a pressure $\pi \in \mc D'([0, T[ \times \R^d)$. Then the following assertions are equivalent:
\begin{enumerate}[(i)]
\item the flow $u$ solves the projected problem with initial datum $u(0)$ that satisfies $u(0) - u_0 \in {\rm Cst}$,
\item for all times $t \in [0, T[$, we have $u(t) - u(0) \in \mc S'_h$,
\item for all times $t \in [0, T[$, we have $u(t) - u(0) \in {\rm BMO}^{-1}$,
\item the pressure satisfies $\pi \in C^0(]0, T[ ; {\rm BMO})$,
\item the pressure force is continuous with respect to time $\nabla \pi \in C^0(]0, T[; \mc S')$ and $\nabla \pi(t) \in \mc S'_h$ for all $0 < t < T$.
\item the pressure defines a continuous function in the topology of locally integrable functions modulo constants $\pi \in C^0(]0, T[ ; L^1_{\rm loc}/ \, \R)$, and we have $\chi(D) \pi(t,x) = O \big( \log |x| \big)$ as $|x| \rightarrow + \infty$ for all $0 < t < T$.
\end{enumerate}
\end{thm}

\begin{rmk}
In this Theorem, the velocity field is assumed to be continuous with respect to time in the $L^\infty$ topology. This may seem odd when compared to known solutions of the Euler equations. For example, the solutions of Pak and Park \cite{PP} are continuous in $B^1_{\infty, 1}$, and therefore, by Proposition \ref{p:c2Gamma},
\begin{equation}\label{eq:c2ExchReg}
\partial_t u = - \P \D (u \otimes u) \in C^0(B^0_{\infty, 1}),
\end{equation}
so that a $C^1(L^\infty)$ assumption might seem more natural, in addition to (as we will see) simplifying the proof. Unfortunately, exchanging space regularity for time regularity as we just did, requires having a solution of the projected problem. If $u \in C^0_T(B^1_{\infty, 1})$ is only a solution of the Euler equations \eqref{eq:c2Euler}, one cannot repeat operation \eqref{eq:c2ExchReg}, as the pressure force $-\nabla \pi$ may well not be regular with respect to time. Recall that the uniform flow \eqref{ieq:c2Unif} solves the Euler problem even if $f(t)$ is not $C^1$.

In that regard, working with $C^1_T(L^\infty)$ solutions would be somewhat artificial. In contrast, continuity in the time variable is a very reasonable assumption, in the light of the existence of paradoxical solutions (that dissipate kinetic energy) in the class $C^0(L^2)$, see \cite{DS}. In other words, there is no point in going below time-continuous regularity.
\end{rmk}

\begin{rmk}
We note that condition \textit{(ii)} of the theorem is in fact a Galilean condition: its validity is independent of the inertial reference frame in which the velocity is computed. This shows that the equivalence issue has something to do with a fundamental property of the solution, and is not a mere artifact of the Galilean nature of the equations.
\end{rmk}

\subsection{Equivalence of the Two Formulations: Smooth in Time Solutions}\label{ss:c2C1proof}

In this subsection, we prove Theorem \ref{t:c2Euler} under a comfortable assumption that the solutions are $C^1(L^\infty)$. This simplifies the proof very much, as it will allow the derivative $\partial_t u$ to be handled as a locally integrable function.

\begin{prop}\label{p:c2Euler1}
Consider $T > 0$. Let $u_0 \in L^\infty$ be a divergence-free initial datum and $u \in C^1([0, T[;L^\infty)$ be a weak solution of the Euler problem (according to Definition \ref{d:c2weakEuler}) for some pressure field $\pi \in \mc D'([0, T[ \times \R^d)$ and related to the initial datum $u_0$. Then all conditions of Theorem \ref{t:c2Euler} are equivalent.
\end{prop}

The proof of this proposition is split in several steps. We start by reformulating the problem in a way that lets the Leray projector appear in the equations. This is the purpose of the following lemma.

\begin{lemma}\label{l:proofL1}
Let $u$ be as in Proposition \ref{p:c2Euler1}. For all times $t \in ]0, T[$, there exists a polynomial $Q(t) \in \R [X]$ such that
\begin{equation}\label{eq:L1}
\partial_t u + \P \D (u \otimes u) + \nabla Q(t) = 0.
\end{equation}
Moreover, $Q(t)$ must be at most one linear: $\deg \big( Q(t) \big) \leq 1$.
\end{lemma}

\begin{rmk}\label{r:EulerGalilean}
A consequence of Lemma \ref{l:proofL1} is that any solution of the Euler system \eqref{eq:c2Euler} is obtained from a solution of the projected system \eqref{eq:c2EulerP} from a ``generalized Galilean transformation'', or in other words  writing the projected equations \eqref{eq:c2EulerP} in an accelerated reference frame. It is mainly a matter of taking $g(t) = \nabla Q(t)$ in \eqref{ieq:c2GenGalInv}. We refer to \cite{Kukavica} and \cite{KV} for precise arguments.
\end{rmk}

\begin{proof}[Proof of Lemma \ref{l:proofL1}]
Since $u \in C^1_T(L^\infty)$ is a solution of the Euler equations for the pressure $\pi$, we may write
\begin{equation*}
\partial_t u + \D(u \otimes u) + \nabla \pi = 0 \qquad \text{in } \mc D'(]0, T[ \times \R^d).
\end{equation*}
Now we fix once and for all a time $t \in ]0, T[$. By taking the divergence of this equation, we see that the pressure solves the elliptic equation
\begin{equation*}
- \Delta \pi = \partial_j \partial_k (u_j u_k),
\end{equation*}
in which there is an implicit sum on the repeated indices. We wish to invert this equation in order to recover $\pi$. In order to avoid the singularity of the inverse Laplacian operator $(- \Delta)^{-1}$ at the frequency $\xi = 0$, we localize away from low frequencies. For any $\phi \in \mc S$ such that $\what{\phi} \in \mc D (\R^d \setminus \{ 0 \})$, we have
\begin{equation*}
\big\langle \nabla \pi, \phi \big\rangle_{\mc S' \times \mc S} = \big\langle \nabla (- \Delta)^{-1} \partial_j \partial_k (u_j u_k), \phi \big\rangle_{\mc S' \times \mc S},
\end{equation*}
where the quantity $\nabla (- \Delta)^{-1} \partial_j \partial_k (u_j u_k) = T_{j, k} (u_j, u_k)$ it to be understood in the sense of Proposition \ref{p:c2Gamma}. Therefore, the difference between the two distributions $\nabla \pi$ and $T_{j,k}(u_j u_k)$ is spectrally supported at $\xi = 0$, and must be a polynomial: there exists a $\nabla Q(t) \in \R[X]$ such that 
\begin{equation*}
\nabla \pi = - ({\rm Id} - \P) \D (u \otimes u) + \nabla Q,
\end{equation*}
and the polynomial $\nabla Q$ is indeed a gradient function because all other terms in this equation are gradients. We have shown that \eqref{eq:L1} holds. The assertion on the degree of $Q(t)$ is obtained by applying the low frequency block $\Delta_{-1}$ to \eqref{eq:L1} and noting that
\begin{equation*}
-\nabla Q = \partial_t \Delta_{-1} u + \Delta_{-1}\P \D (u \otimes u) \in C^0_T(L^\infty)
\end{equation*}
is at all times bounded in view Proposition \ref{p:c2Gamma}, thus ending the proof of the lemma.
\end{proof}

Before moving onwards, we study the Leray projection more closely. The following proposition shows that, for any $f \in L^\infty$, the function $({\rm Id} - \P) \D (f)$ is in $\mc S'_h$.

\begin{prop}\label{p:OpIsHom}
Let $\Gamma$ be as defined in Proposition \ref{p:c2Gamma}. Then we have, for all $p \in [1, +\infty]$ and all $f \in L^p$, the estimate
\begin{equation*}
\left\| \chi(\lambda D) \big( \Gamma * f \big) \right\|_{L^p} = O \left( \frac{1}{\lambda} \log (\lambda) \right) \qquad \text{as } \lambda \rightarrow +\infty.
\end{equation*}
The factor $\log(\lambda)$ may be dispensed with whenever $d \geq 3$.
\end{prop}

\begin{rmk}
As mentioned above, the $O \left( \lambda^{-1} \right)$ estimate (neglecting the logarithmic term when $d=2$) seems rather natural, considering that convolution by $\Gamma$ is formally the Fourier multiplication by a homogeneous symbol of order $1$. However, the first Bernstein inequality is insufficient to obtain this result, as, though it may be generalized to Fourier multipliers, it requires their symbol to be smooth.
\end{rmk} 

\begin{rmk}
In fact, Proposition \ref{p:OpIsHom} could be replaced by a much shorter but much more advanced argument: assume that $f \in L^p$ with $p > 1$. If $1 < p < + \infty$ then Calder\'on-Zygmund theory implies that $\Gamma * f$ is the derivative of a $L^p$ function, so we may legitimately use the first Bernstein inequality to get
\begin{equation*}
\left\| \chi(\lambda D) \big( \Gamma * f \big) \right\|_{L^p} = O \left( \frac{1}{\lambda^{1 + d/p}} \right) \qquad \text{as } \lambda \rightarrow +\infty,
\end{equation*}
but this argument fails for the endpoint exponents, and in particular for $p = + \infty$ which corresponds to our framework for solutions. When $p = + \infty$, we may use the other properties of Singular Integral Operators to see that the map $(- \Delta)^{-1} \partial_j \partial_k \partial_l : L^\infty \tend {\rm BMO}^{-1}$ is bounded. Since the space ${\rm BMO}^{-1}$ is a subspace of $\mc S'_h$ by Proposition \ref{p:c1KockTataruSpH}, we may conclude that $\Gamma * f \in \mc S'_h$. However, we prefer giving an elementary proof based on an integral decomposition of the kernel $\Gamma$ rather than resorting to Fefferman-Stein duality.
\end{rmk}

\begin{proof}[Proof of Proposition \ref{p:OpIsHom}]
We recall that the Fourier multiplier $\Delta_{-1} = \chi (D)$ is equal to the convolution operator $f \mapsto \psi * f$ (see Proposition \ref{p:c2Gamma}). Therefore, we have, for any $\lambda > 0$,
\begin{equation*}
\chi (\lambda D) = \psi_\lambda * , \qquad \text{ where } \psi_\lambda(x) = \frac{1}{\lambda^d} \psi \left( \frac{x}{\lambda} \right).
\end{equation*}
Fix an exponent $\alpha > 0$ and define, for all $\lambda > 0$, the cut-off function $\theta_\lambda (x) = \chi (\lambda^{- \alpha}x)$. By proceeding as in the proof of Proposition \ref{p:c2Gamma}, we have, for $\lambda$ large enough that $\chi(\lambda D)\Delta_{-1} = \chi (\lambda D)$,
\begin{equation*}
\chi(\lambda D) \Gamma = \psi_\lambda * \partial_j \partial_k \partial_l \big( (1 - \theta_\lambda )E \big) + \partial_j \partial_k \partial_l \psi_\lambda * (\theta_\lambda E).
\end{equation*}

To prove the proposition, the Hausdorff-Young convolution inequality asserts that it is enough to show that the $L^1$ norm of both these convolution products is $O(\lambda^{-1} \log(\lambda))$. We start by studying the second one: 
\begin{equation*}
\left\| \partial_j \partial_k \partial_l \psi_\lambda * (\theta_\lambda E) \right\|_{L^1} \leq \left\| \partial_j \partial_k \partial_l \psi_\lambda \right\|_{L^1} \| \theta_\lambda E \|_{L^1}.
\end{equation*}
On the one hand, the three derivatives $\partial_j \partial
_k \partial_l$ provide a $\lambda^{-3}$ decay, as
\begin{equation*}
\left\| \partial_j \partial_k \partial_l \psi_\lambda \right\|_{L^1} = \frac{1}{\lambda^3} \int \left| \left( \partial_j \partial_k \partial_l \psi \right) \left( \frac{x}{\lambda} \right) \right| \frac{\dx}{\lambda^d} = O \left( \frac{1}{\lambda^3} \right).
\end{equation*}
On the other hand, the fundamental solution $E(x)$ is locally integrable, because it has a $O (|x|^{d-2})$ singularity at $x = 0$ when $d \geq 3$, and a $O(\log |x|)$ one if $d=2$. Since the function $\theta_\lambda$ is supported in the ball $B(0, 2\lambda^\alpha)$, we have another inequality: in the case where $d \geq 3$, we get
\begin{equation*}
\| \theta_\lambda E \|_{L^1} \leq C \int_{|x| \leq \lambda^\alpha} \frac{\dx}{|x|^{d-2}} = O(\lambda^{2 \alpha}),
\end{equation*}
and in the case where $d = 2$, we instead have
\begin{equation*}
\begin{split}
\| \theta_\lambda E \|_{L^1} \leq C \int_{|x| \leq \lambda^\alpha}  \log |x| \dx & \leq C \int_0^{\lambda^\alpha} r \log(r) {\rm d} r \\
& = \frac{C}{2} r^2 \Big( \log(r) - \frac{1}{2} \Big) \bigg|_{r=0}^{\lambda^\alpha} = O \left( \lambda^{2 \alpha} \log(\lambda) \right).
\end{split}
\end{equation*}
We conclude that, in all dimensions $d \geq 2$, the convolution product $\partial_j \partial_k \partial_l \psi_\lambda * (\theta_\lambda E)$ has a $L^1$ norm that tends to zero as long as $\alpha < 3/2$, at the speed $O(\lambda^{2 \alpha - 3} \log(\lambda) )$.

\medskip

We next look at the other convolution product $\psi_\lambda * \partial_j \partial_k \partial_l \big( (1 - \theta_\lambda )E \big)$. Here, we will take advantage of the integrability the third derivatives of $E(x)$ possess at $|x| \rightarrow +\infty$. However, this in itself is not enough, as we aim at showing decay as $\lambda \rightarrow + \infty$. We will also have to use the fact that the support of the cutoff $1 - \theta_\lambda$ shrinks as $\lambda$ becomes large. More precisely, we have the estimate, which holds in any dimension $d \geq 2$,
\begin{equation*}
\left| \partial_j \partial_k \partial_l \big( (1 - \theta_\lambda)E \big) (x) \right| \leq C \frac{1 - \mathds{1}_{B(0, 2 \lambda^\alpha)} (x)}{|x|^{d+1}} := M_\lambda (x).
\end{equation*}
By using this on the convolution product, we find that
\begin{equation*}
\left\| \psi_\lambda * \partial_j \partial_k \partial_l \big( (1 - \theta_\lambda )E \big) \right\|_{L^1} \leq C \| \psi_\lambda \|_{L^1} \| M_\lambda \|_{L^1} \leq C \| M_\lambda \|_{L^1}.
\end{equation*}
Finally, we may bound this last integral by
\begin{equation*}
\| M_ \lambda \|_{L^1} \leq C \int_{|x| \geq \lambda^\alpha} \frac{\dx}{|x|^{d+1}} = C \int_{\lambda^\alpha}^{+\infty} \frac{{\rm d}r}{r^2} = O \left( \frac{1}{\lambda^\alpha} \right).
\end{equation*}

\medskip

Putting both estimates together, we have a low frequency inequality for the kernel $\Gamma$. In the limit $\lambda \rightarrow + \infty$, 
\begin{equation*}
\|\chi (\lambda D) \Gamma \|_{L^1} = O \left( \frac{1}{\lambda^\alpha} + \frac{1}{\lambda^{3 - 2 \alpha}} \log (\lambda) \right),
\end{equation*}
and this gives convergence to $0$ as long as $0 < \alpha < 3/2$. Taking $\alpha = 1$ (the optimal value) ends proving our statement.
\end{proof}

Proposition \ref{p:OpIsHom} has the following consequence.

\begin{cor}\label{c:Leray}
Let $f \in B^0_{\infty, \infty} (\R^d ; \R^d \otimes \R^d)$ be a field of matrices. Then we have
\begin{equation*}
\P \D (f) = \sum_{m \in \mathbb{Z}} \dot{\Delta}_m \P \D (f) \qquad \text{ with convergence in } \mc S'.
\end{equation*}
We therefore define a bounded operator $\P \D : B^0_{\infty, \infty} \tend \mathfrak{B}^{-1}_{\infty, \infty} \subset \mc S'_h$. Here, $\mathfrak{B}^{-1}_{\infty, \infty}$ refers to the realization of $\dot{B}^{-1}_{\infty, \infty}$ as a subspace of $\mc S'_h$, see Definition \ref{d:c2BesovRealizSpH}. The same statement holds for the operator $\mathbb{Q} \D = ({\rm Id} - \P)\D$.
\end{cor}

\begin{proof}
Let $A$ be the operator $A = \P \D$ or $A = \Q \D$. Because the symbol of $A$ is a homogeneous function of degree one, Lemma \ref{l:c2FourierMultiplier} makes it clear that 
\begin{equation*}
\forall f \in B^0_{\infty, \infty}, \qquad \| Af \|_{\dot{B}^{-1}_{\infty, \infty}} \leq C \| f \|_{B^0_{\infty, \infty}}.
\end{equation*}
In addition, Proposition \ref{p:OpIsHom} shows that for any $f \in B^0_{\infty, \infty}$, the function $T_{j, k, l}f = (- \Delta)^{-1} \partial_j \partial_k \partial_l f$ lies in $\mc S'_h$, and the first Bernstein inequality (Lemma \ref{l:c2bern}) that
\begin{equation*}
\| \chi(\lambda D) \D (f) \|_{L^\infty} = \| \chi(\lambda D) \Delta_{-1} \D (f) \|_{L^\infty} = O \left( \frac{1}{\lambda} \right) \qquad \text{as } \lambda \rightarrow + \infty.
\end{equation*}
Therefore $Af$ lies in $\mc S'_h$.
\end{proof}

We now have all the necessary elements to prove Proposition \ref{p:c2Euler1}.

\begin{proof}[Proof of Proposition \ref{p:c2Euler1}]
With Lemma \ref{l:proofL1} and Corollary \ref{c:Leray} at our disposal, we are ready to complete the proof of Proposition \ref{p:c2Euler1}, as it is equivalent for $u$ to solve the projected problem \eqref{eq:c2EulerP} and for the polynomial $\nabla Q$ in Lemma \ref{l:proofL1} to be zero. Now, by Corollary \ref{c:Leray}, we see that $\partial_t u$ decomposes as
\begin{equation*}
\partial_t u = - \P \D (u \otimes u) - \nabla Q \in \mc S'_h \oplus \R[X],
\end{equation*}
and so $\nabla Q \equiv 0$ if and only if $\partial_t u (t) \in \mc S'_h$ at all times $t \in ]0, T[$. 

\medskip

We start by showing that assertions \textit{(i)} and \textit{(ii)} of Theorem \ref{t:c2Euler} are equivalent: assume that condition \textit{(ii)} holds, so that $u(t) - u(0) \in \mc S'_h$ for all $t \in [0, T[$. By differentiating with respect to time, we see that $\partial_t u (t)$ is also in $\mc S'_h$ for all $t \in ]0, T[$. Indeed, since $u \in C^1_T(L^\infty)$ the difference quotients $h^{-1}(u(t+h) - u(t))$ will converge to its time derivative $\partial_t u(t)$ for the norm topology of $L^\infty$, and the fact that the difference quotients are in $\mc S'_h \cap L^\infty$ insures that $\partial_t u(t)$ also is, because $\mc S'_h \cap L^\infty$ is a closed subspace of $L^\infty$ (see Proposition \ref{p:c2LinftyHClosed}). This implies that $\nabla Q \equiv 0$ and that $u$ solves the projected problem \eqref{eq:c2EulerP} on $]0, T[ \times \R^d$. 

The fact that it does so with initial datum $u(0)$, and in the sense of Definition \ref{d:c2EulerP}, simply stems from the $C^1$ regularity with respect to time. 

Finally, we show that $V = u_0 - u(0)$ is a constant function.  Any solution of the projected problem \eqref{eq:c2EulerP} must also be a solution of the Euler problem \eqref{eq:c2Euler} with the same initial datum. Since $u$ solves both problems (in the sense of Definitions \ref{d:c2weakEuler} and \ref{d:c2EulerP}), this implies that the weak form \eqref{eq:c2weakMomentum} of the momentum equation must hold for both initial data, so
\begin{equation*}
\int \Big( u_0 - u(0) \Big) \cdot \phi \dx = 0
\end{equation*}
for all divergence-free $\phi \in \mc D (\R^d ; \R^d)$, and so $V \in L^\infty$ is the gradient of some tempered distribution $V = \nabla h$. But since $V$ is also divergence-free, $h$ must be a harmonic polynomial, so that $V$ is also polynomial. The fact that $V \in L^\infty$ forces $V$ to be constant. We have shown \textit{(ii)} to be true.

\medskip

Next, we instead assume that condition \textit{(i)} holds. By integrating with respect to time, we have
\begin{equation*}
u(t) - u(0) = \int_0^t \partial_t u(s) {\rm d}s,
\end{equation*}
this integral being well defined as an element of the Banach space $L^\infty$ (\textsl{e.g.} as a limit of Riemann sums). The fact that $\mc S'_h \cap L^\infty$ is closed in $L^\infty$ for the strong topology implies that we must also have $u(t) - u(0) \in \mc S'_h$.

\medskip

Condition \textit{(iii)} implies \textit{(ii)} because ${\rm BMO}^{-1} \subset \mc S'_h$. And if \textit{(i)} is true, then the difference $u(t) - u(0)$ is given by
\begin{equation*}
u(t) - u(0) = \int_0^t \partial_t u(s) {\rm d} s = - \int_0^t \P \D (u \otimes u) {\rm d} s.
\end{equation*}
Because $u \in C^0([0, T[; L^\infty)$, the function $\P \D (u \otimes u)$ lies in $C^0([0, T[; {\rm BMO}^{-1})$. Therefore the integrals above are well defined (as limits of Riemann sums) in the Banach space ${\rm BMO}^{-1}$, thus implying \textit{(iii)}.

\medskip

Concerning conditions \textit{(iv)}, \textit{(v)} and \textit{(vi)}, if aassetion \textit{(i)} holds, then the pressure is given by Leray projection in the form of a singular integral operator
\begin{equation*}
\pi = (- \Delta)^{-1} \partial_j \partial_k (u_j u_k) \qquad \text{in } C^0(]0, T[; {\rm BMO}).
\end{equation*}
In the above, time $t=0$ is excluded: because the initial datum $u_0$ may be differ from the initial flow $u(0)$ by a constant, the pressure may exhibit a singularity at initial time
\begin{equation*}
\nabla \pi = - \nabla (- \Delta)^{-1} \partial_j \partial_k (u_j u_k) + \delta_0 \otimes \big( u(0) - u_0 \big).
\end{equation*}
At any rate, we deduce that both \textit{(iv)} and \textit{(v)} are true, and point \textit{(vi)} follows from Proposition \ref{p:logBMO}. Finally, assume \textit{(v)}, which is the weakest of the three conditions \textit{(iv)}, \textit{(v)} and \textit{(vi)} by Lemma \ref{l:c2KukavicaVicolCondition} and Proposition \ref{p:logBMO}. Then, for every time $t \in ]0, T[$, the derivative $\partial_t u(t)$ must lie in $\mc S'_h$. By the discussion above concerning points \textit{(i)} and \textit{(ii)}, it appears that the polynomial $\nabla Q(t)$ must be supported at time $t = 0$ and that the solution fulfills \textit{(i)}. 
\end{proof}

\subsection{Equivalence of the Two Formulations: Low Time Regularity}

In this paragraph, we fully prove Theorem \ref{t:c2Euler} by working with the more general class of $C^0_T(L^\infty)$ solutions of \eqref{eq:c2Euler}. The obvious problem is that the derivative $\partial_t u$ might not exist as a function of time for these solutions. We will use a regularization procedure to circumvent this issue.

\begin{proof}[Proof of Theorem \ref{t:c2Euler}]
Because we deal with solutions which have low time regularity, we take a mollification sequence $\big( K_\epsilon (t) \big)_{\epsilon > 0}$ such that $K_1$ is supported in the compact interval $[-1, 1]$, $\| K_1 \|_{L^1} = 1$ and
\begin{equation*}
K_\epsilon (t) = \frac{1}{\epsilon} K_1 \left( \frac{t}{\epsilon} \right).
\end{equation*}
The function $K_1$ is nonnegative by definition. Next, we extend all functions $u$ and $\pi$ to functions on the \textit{open} set $t \in ]- \infty, T[$ by setting them to zero on $]- \infty, 0[$. For the sake of simplicity, we continue to note $u$ and $\pi$ the extensions. This allows us to incorporate the initial data condition in the righthand side of the equations, and see that $(u, \pi)$ solves
\begin{equation*}
\partial_t u + \D (u \otimes u) + \nabla \pi = \delta_0(t) \otimes u_0 (x) \qquad \text{in } \mc D' \big( ]- \infty, T[ \times \R^d \big).
\end{equation*}
For all $\epsilon > 0$, we note $(u_\epsilon, \pi_\epsilon) = K_\epsilon * (u, \pi)$, where the convolution is done with respect to the time variable $t \in ]- \infty, T - \epsilon [$, as the kernel $K_\epsilon$ is supported in $[- \epsilon, \epsilon]$. We obtain, by convoluting with respect to time, an equality in the sense of $\mc D' (]- \infty, T - \epsilon[ \times \R^d)$:
\begin{equation}\label{NVproofEQ1}
\begin{split}
\partial_t u_\epsilon + K_\epsilon * \D (u \otimes u) + \nabla \pi_\epsilon & = (K_\epsilon * \delta_0) (t) \otimes u_0 (x)\\
& = K_\epsilon (t) u_0(x).
\end{split}
\end{equation}
By taking the divergence of this equation and the $m$-th homogeneous dyadic block $\dot{\Delta}_m$, we see that
\begin{equation*}
\forall m \in \mathbb{Z}, \qquad \dot{\Delta}_m \nabla \pi_\epsilon = \sum_{j, k} \dot{\Delta}_m \nabla (- \Delta)^{-1} \partial_j \partial_k K_\epsilon * (u_j u_k).
\end{equation*}
Because, for all $t \in ]- \infty, T - \epsilon [$, the functions $K_\epsilon * (u_j u_k) (t)$ are bounded, Corollary \ref{c:Leray} insures that we can sum the previous equation over $m \in \mathbb{Z}$ to get
\begin{equation*}
\nabla \pi_\epsilon = \nabla Q_\epsilon (t) - K_\epsilon * ({\rm Id} - \P) \D (u \otimes u),
\end{equation*}
for some polynomial $Q_\epsilon (t) \in \R[X]$. By substituting this expression in \eqref{NVproofEQ1}, we find that
\begin{equation}\label{eq:nonRegTimeEQ1}
\partial_t u_\epsilon + K_\epsilon * \P \D (u \otimes u) + \nabla Q_\epsilon = K_\epsilon u_0, \qquad \text{for } t \in ]- \infty, T - \epsilon [.
\end{equation}
By taking the limit $\epsilon \rightarrow 0^+$ in the previous equation, we see that all terms converge in $\mc D' (]- \infty, T[ \times \R^d)$ to some limit.\footnote{For any fixed value of $\epsilon > 0$, the different terms do not really define a distribution on $]- \infty, T[ \times \R^d$, but duality bracket involving a test function $\phi \in \mc D (]- \infty, T[ \times \R^d)$ will always be defined for $\epsilon$ small enough.} Therefore, the $\nabla Q_\epsilon$ must also have a limit as $\epsilon \rightarrow 0^+$, which we will note $\nabla Q$, and which satisfies
\begin{equation*}
\partial_t u + \P \D (u \otimes u) + \nabla Q = \delta_0 (t) \otimes u_0 (x) \qquad \text{in } \mc D' \big( ]- \infty, T[ \times \R^d \big).
\end{equation*}
By taking the time convolution of this last equation by $K_\epsilon (t)$, we see that we in fact have $\nabla Q_\epsilon = K_\epsilon * \nabla Q$. The whole proof hinges on finding under what condition $\nabla Q = 0$ in $\mc D' (]0, T[ \times \R^d)$.

\begin{rmk}
In equation \eqref{eq:nonRegTimeEQ1}, the polynomial $Q_\epsilon(t) \in \R[X]$ is defined for every time $t \in ]- \infty, T - \epsilon[$. This choice makes sense because all functions implied in \eqref{eq:nonRegTimeEQ1} are $C^\infty$ smooth with respect to time. However, $Q_\epsilon$ has no reason to define a distribution on $]- \infty, T - \epsilon[$. It is determined by $\pi_\epsilon$ (which is, by assumption, a distribution) \textsl{up to a constant summand} $C(t)$ which may be very singular, $C(t) = |t|^{-1}$ for example, or not even measurable. 

By contrast, the derivative $\nabla Q_\epsilon$ is a distribution on $]- \infty, T - \epsilon[$, and converges to a distribution $\nabla Q$ in the limit $\epsilon \rightarrow 0^+$. This convergence does not concern $Q_\epsilon$, so $Q$ is not \textsl{per se} a well defined object. But this poses no problem as we will only work with the (formal) derivative $\nabla Q$.
\end{rmk}

\medskip

We start by assuming that condition \textit{(i)} holds, so that $\nabla Q = 0$ on $]0, T[ \times \R^d$. This implies that $\nabla Q_\epsilon (t) = 0$ for $t \in [\epsilon, T - \epsilon]$, since $K_\epsilon$ is compactly supported in $[- \epsilon, \epsilon]$. Similarly, the term containing the initial datum $u_0$ in \eqref{NVproofEQ1} vanishes as soon as $t \geq \epsilon$. Therefore, for $t \in [\epsilon, T - \epsilon]$,
\begin{equation*}
\partial_t u_\epsilon (t) = - K_\epsilon * \P \D (u \otimes u) \in \mc S'_h.
\end{equation*}
Therefore, by integrating with respect to time, we find that
\begin{equation*}
u_\epsilon (t) - u_\epsilon (\epsilon) = \int_\epsilon^t \partial_t u_\epsilon (s) \text{\rm d} s \in L^\infty \cap \mc S'_h.
\end{equation*}
Since $u_\epsilon \in C^\infty(L^\infty)$, this last integral is well-defined as an element of the Banach space $L^\infty \cap \mc S'_h$ (see Proposition \ref{p:c2LinftyHClosed}) as, \textsl{e.g.} a limit of Riemann sums. We wish to take the limit $\epsilon \rightarrow 0^+$ in this last equation and use the fact that $L^\infty \cap \mc S'_h$ is closed in $L^\infty$.

Thanks to the convergence $u_\epsilon \tend u$ in the space $C^0(]0, T[ ; L^\infty)$, which is equipped with the topology of uniform convergence on compact sets of $]0, T[$, we deduce on the one hand that we have pointwise convergence $u_\epsilon (t) \tend u(t)$. On the other hand, to compute the limit of $u_\epsilon(\epsilon)$, we write
\begin{equation*}
\begin{split}
\| u_\epsilon (\epsilon) - u(0) \|_{L^\infty} & \leq \| u_\epsilon (\epsilon) - u (\epsilon) \|_{L^\infty} + \| u (\epsilon) - u(0) \|_{L^\infty} \\
& = \| u_\epsilon (\epsilon) - u (\epsilon) \|_{L^\infty} + o(1).
\end{split}
\end{equation*}
Next, by writing explicitly the convolution integrals involved in $u_\epsilon (\epsilon) - u (\epsilon)$ and using the fact that the kernel $K_\epsilon$ is supported in $[-\epsilon, \epsilon]$, we get that
\begin{equation*}
u_\epsilon (\epsilon) - u (\epsilon) = \int_0^{2 \epsilon} K_\epsilon (\epsilon - s) u(s) {\rm d} s - u(\epsilon) = \int_0^{2 \epsilon} K_\epsilon (\epsilon - s) \Big\{ u(s) - u(\epsilon) \Big\} {\rm d} s,
\end{equation*}
The function $u \in C^0([0, T[ ; L^\infty)$ is uniformly continuous on any compact interval $[0, \eta]$. Therefore, we may take the $L^\infty$ norm of the above equation to obtain
\begin{equation*}
\| u_\epsilon (\epsilon) - u (\epsilon) \|_{L^\infty} \leq \| K_\epsilon \|_{L^1} \sup_{s \in [0, 2 \epsilon]} \| u(s) - u(\epsilon) \|_{L^\infty} \tend 0 \qquad \text{as } \epsilon \rightarrow 0^+.
\end{equation*}
This proves that $u_\epsilon (t) - u_\epsilon (\epsilon)$ converges to $u(t) - u(0)$ as $\epsilon \rightarrow 0^+$. Because $L^\infty \cap \mc S'_h$ is closed in $L^\infty$, we finally deduce that $u(t) - u(0) \in \mc S'_h$ for all $0 \leq t < T$.

\medskip

Now, let us assume that $u(t) - u(0) \in \mc S'_h$ for all $t \in [0, T[$. We take the convolution of $u(t) - u(0)$ with $K_\epsilon$ to find that $u_\epsilon(t) - u(0)$ also lies in $L^\infty \cap \mc S'_h$, but only for $t \in ]\epsilon, T - \epsilon[$, as the condition $u(t) - u(0) \in \mc S'_h$ does not hold on the extension for $t \notin [0, T[$, as $u(0)$ need not be in $\mc S'_h$. Since $u_\epsilon (t) - u(0)$ is a $C^1$ function with respect to time, we can differentiate. The derivative $\partial_t u_\epsilon$ will then be found as the $L^\infty$ limit of the difference quotients, and will therefore also be an element of $L^\infty \cap \mc S'_h$,
\begin{equation*}
\forall t \in ]\epsilon, T - \epsilon [, \qquad \partial_t u_\epsilon (t) \in \mc S'_h.
\end{equation*}
This implies that the polynomial $\nabla Q_\epsilon (t)$ from \eqref{eq:nonRegTimeEQ1} must also be in $\mc S'_h$ if $t \in ]\epsilon, T - \epsilon[$, and so must be zero on that time interval. Therefore, $\nabla Q \equiv 0$ in $\mc D' (]0, T[ \times \R^d)$. This means that $u$ solves the projected system \eqref{eq:c2EulerP}, we just have to check the assertion concerning the initial datum for the velocity.

Since the function $u$ solves the projected equation, we may use it to prove that $u$ is in fact $C^1_T(\dot{B}^{-1}_{\infty, \infty})$ regular. The $C^1$ time regularity implies that $u$ will be a weak solution of the projected equation with initial datum $u(0)$, as in Definition \ref{d:c2EulerP}.

Lastly, since $u$ solves the projected problem with initial datum $u(0)$, it must solve the Euler system with the datum $u(0)$. Therefore, $u$ is a weak solution of the Euler equations, as in Definition \ref{d:c2weakEuler} with both initial data $u_0$ and $u(0)$, and so $u_0 - u(0)$ must be a constant function, as shown in Subsection \ref{ss:c2C1proof} above. This shows the equivalence of points \textit{(i)} and \textit{(ii)}.

\medskip

Now that we have established that \textit{(i)} and \textit{(ii)} are equivalent, we may focus on the last three points. Of course, \textit{(iii)} implies \textit{(ii)}. And if $u$ is a solution of the projected problem, so that \textit{(i)} holds, we must have
\begin{equation*}
\partial_t u = - \P \D (u \otimes u) \qquad \text{in } \mc D'(]0, T[ \times \R^d).
\end{equation*}
This implies that $\partial_t u$ is uniformly continuous with respect to $t \in ]0, T[$ in the ${\rm BMO}^{-1}$ topology, so that for any $t \in [\eta, T[$ the following integral is well defined:
\begin{equation*}
u(t) - u(\eta) = - \int_\eta^t \P \D (u \otimes u) {\rm d}s \tend_{\eta \rightarrow 0^+} \; - \int_0^t \P \D (u \otimes u) {\rm d}s \qquad \text{in } {\rm BMO}^{-1}.
\end{equation*}
Because $u \in C^0([0, T[; L^\infty)$, the lefthand side of this equation converges (in $L^\infty$) to $u(t) - u(0)$ as $\eta \rightarrow 0^+$. Uniqueness of the limit implies that \textit{(iii)} holds.

\medskip

Finally, if $u$ is a solution of the projected problem (that is, if \textit{(i)} is true), then we may conclude to \textit{(iv)}, \textit{(v)} and \textit{(vi)} exactly as in the proof of Proposition \ref{p:c2Euler1}. On the other hand, if \textit{(v)} is true (recall that it is the weaksest of the three conditions \textit{(iv)}, \textit{(v)} and \textit{(vi)} by Lemma \ref{l:c2KukavicaVicolCondition} and Proposition \ref{p:logBMO}), then we check that $\nabla \pi_\epsilon (t) \in \mc S'_h$. For any test function $\phi \in \mc S$ and any $t \in ]\epsilon, T - \epsilon[$,
\begin{equation*}
\big\langle \chi(\lambda D) \nabla \pi_\epsilon, \phi \big\rangle_{\mc S' \times \mc S} = \int K_\epsilon (t - s) \big\langle \chi(\lambda D) \nabla \pi(s), \phi \big\rangle_{\mc S' \times \mc S} \, {\rm d} s.
\end{equation*}
Because $\nabla \pi \in C^0(]0, T[; \mc S')$, the bracket in the integral is a continuous function of time that converges to zero as $\lambda \rightarrow + \infty$ (by assumption) while remaining uniformly bounded. Therefore, dominated convergence guarantees that the whole integral tends to zero as $\lambda \rightarrow + \infty$, hence $\nabla \pi_\epsilon (t) \in \mc S'_h$. With that remark in mind, bringing the proof to its conclusion is only a matter of noticing that
\begin{equation*}
\nabla \pi_\epsilon = K_\epsilon * \nabla (- \Delta)^{-1} \partial_j \partial_k (u_j u_k) + \nabla Q_\epsilon \in \mc S'_h \oplus \R[X].
\end{equation*}
Since $\nabla \pi_\epsilon \in \mc S'_h$, the polynomial must be zero $\nabla Q_\epsilon = 0$ and we deduce that $\nabla Q = 0$ on $]0, T[ \times \R^d$.
\end{proof}

\section{Application to the Euler problem}

In this short paragraph, we provide two applications of Theorem \ref{t:c2Euler} to, firstly, a well-posedness result in Besov spaces and, secondly, Serfati solutions of the Euler system.

\subsection{A Full Well-Posedness Result}

Theorem \ref{t:c2Euler} above provides us with a full well-posedness result for the Euler system in the space $C^0_T(B^1_{\infty, 1})$.

\begin{cor}
Consider $u_0 \in B^1_{\infty, 1}$ a divergence-free initial datum. There exists a time $T > 0$ such that the Euler problem \eqref{eq:c2Euler} has a unique solution $u \in C^0 ( [0, T[ ; B^1_{\infty, 1})$ (in the sense of Definition \ref{d:c2weakEuler}) related to the initial datum $u_0$ that satisfies
\begin{equation}\label{eq:c2corCond}
u(0) = u_0 \qquad \text{and} \qquad u(t) - u(0) \in \mc S'_h \text{ for } t \in [0, T[.
\end{equation}
Moreover, this solutions lies in the space $C^1([0, T[ ; B^0_{\infty, 1})$.
\end{cor}

\begin{rmk}
Note that condition \eqref{eq:c2corCond} is a rather natural one for a well-posedness result, as $u(t) - u(0) \in \mc S'_h$ implies that the flow is, unlike the Poiseuille-type flow \eqref{ieq:c2Unif}, not driven by an exterior action (recall that $\mc S'_h$ is a space of functions that are ``on average'' zero at $|x| \rightarrow + \infty$). A flow that is left to its own devices must be deterministic.
\end{rmk}

\begin{proof}
Thanks to Theorem \ref{t:c2Euler}, the proof is straightforward. With the help of the results of \cite{PP}, we have a $T > 0$ such that the projected problem has a unique solution $u \in C^0_T (B^1_{\infty, 1}) \cap C^1_T(B^0_{\infty, 1})$ with initial datum $u_0$. This solution must also solve the original Euler problem with the same initial datum, as well as fulfilling \eqref{eq:c2corCond}.

On the other hand, by Theorem \ref{t:c2Euler}, any solution $v \in C^0_T(B^1_{\infty, 1}) \hookrightarrow C^0_T (L^\infty)$ that fulfills condition \eqref{eq:c2corCond} must also solve the projected problem \eqref{eq:c2EulerP} with the initial datum $u_0$, and so must coincide with $u$ on $[0, T[ \times \R^d$.
\end{proof}

\subsection{Serfati Solutions}\label{ss:c2Serfati}

In this paragraph, we work exclusively on the 2D plane $\R^2$, so that $d = 2$.

\medskip

Serfati solutions of the Euler equations\footnote{Many thanks are due to Raphaël Danchin for introducing us to this topic and to Drago\c s Iftimie for asking the questions studied in this paragraph.} are a class of 2D bounded solutions introduced by Philippe Serfati in \cite{Serfati} (see also \cite{AKLFNL} for an english presentation). These solutions have a threefold description: a solution $u \in C^0(\R_+ ; L^\infty (\R^2))$ is a \textsl{Serfati solution} if
\begin{enumerate}[(i)]
\item both $u$ and the vorticity $\omega = \partial_1 u_2 - \partial_2 u_1$ are uniformly bounded
\begin{equation}\label{eq:c2SerfatiClass}
\sup_{t \geq 0} \Big( \| u \|_{L^\infty} + \| \omega \|_{L^\infty} \Big) < + \infty,
\end{equation}
\item the vorticity solves the usual pure transport equation $\partial_t \omega + u \cdot \nabla \omega = 0$ in the sense of distributions,
\item the velocity $u$ satisfies the \textsl{Serfati identity}: for any radial cut-off function $\theta \in \mc D$ such that $\theta(x) = 1$ around $x = 0$, we have
\begin{equation}\label{eq:c2SerfatiIdentity}
u(t) - u(0) = (\theta K) * \big( \omega(t) - \omega(0) \big) - \int_0^t \nabla \nabla^\perp \big( (1 - \theta)K \big) * (u \otimes u) {\rm d}s,
\end{equation}
where in the equation above $K(x) = x^\perp / |x|^2$ is the Biot-Savart kernel and the convolution product contains an implicit contraction between both $2 \times 2$ matrices $\nabla \nabla^\perp \big( (1 - \theta)K \big)$ and $u \otimes u$.
\end{enumerate}

The Serfati identity follows naturally from the Euler equations in their vorticity form. Formally, the velocity flow satisfies $u = K*\omega$ and hence
\begin{equation*}
\partial_t u = (\theta K) * \partial_t \omega + \big( (1 - \theta) K \big) * \partial_t \omega.
\end{equation*}
On the one hand, the first convolution product $(\theta K)*\omega$ makes sense if $\omega \in L^\infty$ because the Biot-Savart kernel is locally integrable. Integration with respect to time gives the first term in the righthandside of \eqref{eq:c2SerfatiIdentity}. On the other hand, by substituting $\partial_t \omega = - u \cdot \nabla \omega = \curl \D (u \otimes u)$ and integrating by parts, we obtain the second term:
\begin{equation*}
\int_0^t \big( (1 - \theta) K \big) * \partial_t \omega {\rm d}s = -\int_0^t \nabla \nabla^\perp \big( (1 - \theta) K \big) * (u \otimes u) {\rm d}s.
\end{equation*}

\medskip

A notable feature of Serfati solutions is that they lead to \textsl{global a priori estimates} for the quantity $\| u \|_{L^\infty} + \| \omega \|_{L^\infty}$ and to a well-posedness result: for any $u_0 \in L^\infty$ with $\omega_0 = \curl(u_0) \in L^\infty$, there exists a unique Serfati solution with initial datum $u_0$.

As noted in Section 7 of \cite{AKLFNL}, this means that requiring a solution of the Euler system to fulfill the Serfati identity rules out the pathological solutions, such as \eqref{ieq:c2Unif}, that create non-uniqueness. In fact, as is deeply discussed in \cite{Kelliher}, a (weak) solution of the Euler system \eqref{eq:c2Euler} which satisfies \eqref{eq:c2SerfatiClass} is uniquely determined by the initial value and an asymptotic function $U_\infty(t)$ if one of the following sufficient conditions is satisfied:
\begin{enumerate}[(i)]
\item a ``generalized'' Serfati identity holds, where the uniform flow $U_\infty(t)$ is added to \eqref{eq:c2SerfatiIdentity},
\item the velocity satisfies a renormalized Biot-Savart law: if $\theta$ is as above and $\theta_R(x) = \theta(x/R)$, then
\begin{equation*}
(\theta_R K) * \big( \omega(t) - \omega(0) \big) \tend_{R \rightarrow + \infty} U_\infty(t) + u(t) - u(0)
\end{equation*}
locally uniformly in the sense of the norm \eqref{eq:c2SerfatiClass},
\item the pressure $\pi(t, x)$ satisfies 
\begin{equation*}
\pi(t, x) - U_\infty'(t) \cdot x \; \in L^\infty_T({\rm BMO}).
\end{equation*}
\end{enumerate}

\medskip

Here, we show the relation between Serfati solutions and our own result: more precisely, we prove that a Serfati solution fulfills the equivalent conditions of Theorem \ref{t:c2Euler}, and hence is a solution of the projected problem.

\begin{thm}
Let $u \in C^0(\R_+ ; L^\infty)$ be a Serfati solution of the Euler system. Then, for all $\alpha < 1$,
\begin{equation*}
\left\| \chi(\lambda D) \big( u(t) - u(0) \big) \right\|_{L^\infty} = O \left( \frac{1}{\lambda^\alpha} \right) \qquad \text{as } \lambda \rightarrow + \infty.
\end{equation*}
In particular, $u$ is a solution of the projected problem \eqref{eq:c2EulerP} with initial value $u(0)$.
\end{thm}

\begin{rmk}
Serfati solutions are more general than the $B^1_{\infty, 1}$ solutions of Pak and Park \cite{PP}, because $\omega \in L^\infty$ does not imply that $\nabla u \in L^\infty$, as would be the case if we had $u \in B^1_{\infty, 1}$. On the other hand, Serfati solutions form a more restrictive framework for the Euler equations than our simple $C^0(L^\infty)$ context: firstly, Serfati solutions are only defined in 2D, and secondly Theorem \ref{t:c2Euler} holds for very rough solutions, whereas the Serfati identity requires $\omega \in L^\infty$ to be written.
\end{rmk}

\begin{proof}
Let, as usual, $\psi_\lambda \in \mc S$ such that $\what{\psi_\lambda}(\xi) = \chi(\lambda \xi)$ and consider the convolution product of the Serfati identity \eqref{eq:c2SerfatiIdentity}
\begin{equation*}
\psi_\lambda * \big( u(t) - u(0) \big) = \psi_\lambda * (\theta K) * \big( \omega(t) - \omega (0) \big) - \int_0^t \psi_\lambda * \nabla \nabla^\perp \big( (1 - \theta)K \big)*(u \otimes u) {\rm d}s.
\end{equation*}
This equation is true for all nonnegative radial $\theta \in \mc D$ such that $\theta(x) = 1$ near $x=0$. As in the proof of Proposition \ref{p:OpIsHom}, our aim is to estimate the convolution products with the Hausdorff convolution inequality to obtain our result. We focus on the first term. Integration by parts gives
\begin{equation}\label{eq:c2SerfatiEQ1}
\begin{split}
\Big\| \psi_\lambda * (\theta K) * \big( \omega(t) - \omega (0) \big) \Big\|_{L^\infty} & = \left\| \nabla^\perp \psi_\lambda * (\theta K) * \big( u(t) - u(0) \big) \right\|_{L^\infty} \\
& \leq \| \nabla \psi_\lambda \|_{L^1} \| \theta K \|_{L^1} \| u(t) - u(0) \|_{L^\infty} \\
& \leq \frac{1}{\lambda} \| \nabla \psi \|_{L^1} \| \theta K \|_{L^1} \| u(t) - u(0) \|_{L^\infty}.
\end{split}
\end{equation}
We leave here, for now, this first term and come to the second one. Similarly, the Hausdorff convolution inequality gives
\begin{equation*}
\begin{split}
\left\| \psi_\lambda * \nabla \nabla^\perp \big( (1 - \theta)K \big)*(u \otimes u) \right\|_{L^\infty} & \leq \| \psi _\lambda \|_{L^1} \big\| \nabla^2 \big( (1 - \theta)K \big) \big\|_{L^1} \| u \|_{L^\infty}^2 \\
& \leq \big\| \nabla^2 \big( (1 - \theta)K \big) \big\|_{L^1} \| \psi \|_{L^1} \| u \|_{L^\infty}^2,
\end{split}
\end{equation*}
but this is not sufficient: the upper bound we have found does not depend on $\lambda$, so it will not decay as $\lambda \rightarrow + \infty$. To force this upper bound to be small, we take advantage of the fact that the Serfati identity holds for all radial cut-off functions $\theta \in \mc D$. Fix a such $\theta \in \mc D$ that is supported in the ball $|x| \leq 1$ and define $\theta_\lambda(x) = \theta(x/\lambda^\alpha)$ for some $\alpha > 0$ to be fixed later. Then we have
\begin{equation*}
\big\| \nabla^2 \big( (1 - \theta)K \big) \big\|_{L^1} \leq C \int_{\lambda^\alpha}^{+ \infty} \frac{{\rm d}r}{r^2} = O \left( \frac{1}{\lambda^\alpha} \right) \qquad \text{as } \lambda \rightarrow + \infty,
\end{equation*}
and therefore
\begin{equation}\label{eq:c2SerfatiEQ2}
\left\| \int_0^t \psi_\lambda * \nabla \nabla^\perp \big( (1 - \theta)K \big)*(u \otimes u) {\rm d}s \, \right\|_{L^\infty} = O \left( \frac{1}{\lambda^\alpha} \right) \qquad \text{as } \lambda \rightarrow + \infty.
\end{equation}
On the other hand, our choice of cut-off function comes at a cost: because the support of $\theta_\lambda$ grows as $\lambda$ does, we lose decay in \eqref{eq:c2SerfatiEQ1}:
\begin{equation*}
\frac{1}{\lambda} \| \theta_\lambda K \|_{L^1} = O \left( \frac{1}{\lambda^{1 - \alpha}} \right) \qquad \text{as } \lambda \rightarrow + \infty.
\end{equation*}
By putting this together with \eqref{eq:c2SerfatiEQ2} and choosing the optimal $\alpha = 1/2$, we see that $\chi(\lambda D) (u(t) - u(0))$ does indeed tend to zero in $L^\infty$ at the speed $O (\lambda^{-1/2})$. With Theorem \ref{t:c2Euler}, this is enough to prove that $u$ is a solution of the projected problem.

\medskip

However, we are not entirely sated, because the $O(\lambda^{-1/2})$ does not seem optimal: by Proposition \ref{p:OpIsHom}, we know that the decay is truly of order roughly $O (\lambda^{-1})$. By revisiting the computations above, we see that the estimates in \eqref{eq:c2SerfatiEQ1} are not really optimal: because convolution by $\psi_\lambda$ is a low frequency cut-off operator $\chi(\lambda D)$, we have
\begin{equation*}
\begin{split}
\psi_\lambda * (\theta K) * \big( \omega(t) - \omega (0) \big) & = \psi_{2 \lambda} * \psi_\lambda * (\theta K) * \big( \omega(t) - \omega (0) \big) \\
& = \nabla \psi_\lambda * (\theta_\lambda K) * \Big( \psi_{2 \lambda} * (u(t) - u(0)) \Big).
\end{split}
\end{equation*}
Now, thanks to the estimates we have just found, we know that $\psi_{2 \lambda} * (u(t) - u(0))$ also decays in $L^\infty$ at the speed $O (\lambda^{-1/2})$. By taking this into account, we obtain
\begin{equation*}
\Big\| \psi_\lambda * (\theta K) * \big( \omega(t) - \omega (0) \big) \Big\|_{L^\infty} = O \left( \frac{1}{\lambda^{3/2 - \alpha}} \right) \qquad \text{as } \lambda \rightarrow + \infty
\end{equation*}
and by setting $\alpha = 3/4$ we obtain a $O(\lambda^{-3/4})$ estimate instead of a $O(\lambda^{-1/2})$ one, and we can once again use this improved bound in \eqref{eq:c2SerfatiEQ1}... By iterating, we find that the bound may be brought to $O (\lambda^{- \alpha})$ for any $\alpha < 1$, but at the price of a constant that depends on $\alpha$.
\end{proof}

\section{Application to Ideal MHD and Elsässer Variables}

In this Section, we study the ideal magnetohydrodynamic (MHD) equations. These describe ideal conducting fluids which are subject to the magnetic field generated through their own motion. The ideal MHD equations read:
\begin{equation}\label{ieq:c3mhd}
\begin{cases}
\partial_t u + (u \cdot \nabla)u + \nabla \pi = (b \cdot \nabla)b\\
\partial_t b + (u \cdot \nabla)b = (b \cdot \nabla)u\\
\D(u) = 0.
\end{cases}
\end{equation}
In the equations above, $\pi : \R \times \R^d \tend \R$ is the MHD pressure, which is the sum of the usual hydrodynamic pressure and the magnetic pressure $\frac{1}{2}|b|^2$. Formally, the MHD equations are a generalization of the Euler system, as the latter can be recovered by setting $b=0$. However, this obscures the fact that MHD possesses its own interesting and intriguing properties, such as non-trivial static solutions (with $u=0$), additional conserved quantities, peculiar stability issues, a much more intricate theory of solutions, \textsl{etc}. We refer to the excellent book \cite{Bellan} for an overview of some of these issues.

\medskip

Formally, system \eqref{ieq:c3mhd} is overdetermined, as it is a system of $2d+2$ equations whereas there are $2d+1$ scalar unknowns $(u, b, \pi)$. However, the divergence condition $\D (b) = 0$ is in fact redundant, because it automatically follows from the magnetic field equation, provided that the initial datum $\D (b_0) = 0$ is also divergence-free:
\begin{equation*}
\partial_t \D (b) = \partial_k \partial_j \big( u_k b_j - b_k u_j \big) = 0,
\end{equation*}
where, once again, there is an implicit summation on both repeated indices $j, k = 1, ..., d$.

\medskip

By using a suitable change of variables, introduced by the German-American Physicist W. M. Elsässer \cite{Elsasser}, we may see system \eqref{ieq:c3mhd} as a system of transport equations: define the \textsl{Elsässer variables}  by
\begin{equation}\label{eq:RChangeVariables}
\al = u + b \qquad \text{and} \qquad \bt = u - b.
\end{equation}
By adding and substracting the momentum and magnetic field equations, we see that these variables solve the new problem, which we will refer to as the \textsl{Elsässer system},
\begin{equation}\label{eq:elsasser}
\begin{cases}
\partial_t \al + \D (\bt \otimes \al) + \nabla \pi_1 = 0\\
\partial_t \bt + \D (\al \otimes \bt) + \nabla \pi_2 = 0\\
\D (\al) = \D (\bt) = 0.
\end{cases}
\end{equation}
In the previous system, $\pi_1$ and $\pi_2$ are two (\textsl{a priori} distinct) scalar functions which enforce the independent divergence-free conditions $\D (\al) = \D (\bt) = 0$. In contrast to \eqref{ieq:c3mhd}, the Elsässer system \eqref{eq:elsasser} has the same number $2d+2$ of equations as of unknowns, which are $(\al, \bt, \pi_1, \pi_2)$. This difference is explained by noting that the structure of system \eqref{eq:elsasser} implies that we must have, say under mild integrability assumptions (see \cite{CF3} and Theorem \ref{t:c3Els} below), the equality of the ``pressure'' functions $\pi_1 = \pi_2$. This is consistent with the fact that, whenever we sum and substract the first two equations of \eqref{ieq:c3mhd}, we obtain \eqref{eq:elsasser} with $\pi_1 = \pi_2 = \pi + \frac{1}{2}|b|^2$.

As a set of transport equations, the Elsässer system is much easier to deal with than the ``classical'' ideal MHD system, since \eqref{ieq:c3mhd} only is (up to the non-local pressure term) a quasi-linear symmetric hyperbolic system\footnote{In contrast with transport equations, the basic theory of such systems only exists in $L^2$ based spaces. We refer to \cite{Brenner} for more insight on the this: a symmetric hyperbolic system with constant coefficients is well-posed in $L^p$, $p \neq 2$, if and only if the matrix coefficients all commute.}. As a matter of fact, all well-posedness results obtained for the ideal MHD equations have used, in one way or another, this change of variables (see \cite{Sch}, \cite{Secchi}, \cite{MY}, \cite{CF3}).

\medskip

The catch is that, although systems \eqref{ieq:c3mhd} and \eqref{eq:elsasser} are closely related, \textsl{they are not, strictly speaking, equivalent}. In \cite{CF3} (see Section 4), we had shown that equivalence between the two systems holds in the framework of $L^p$ solutions (with $p < +\infty$), and that it does not in for $L^\infty$ solutions. 

Let us explain further how this works: if $(\alpha, \beta)$ is a solution of \eqref{eq:elsasser} with pressures $\pi_1$ and $\pi_2$, then by setting $(u, b) = \frac{1}{2}(\alpha + \beta, \alpha - \beta)$ we construct a solution $u$ of the momentum equation with MHD pressure $\pi = \frac{1}{2}(\pi_1+ \pi_2)$ and a solution $b$ of 
\begin{equation}\label{eq:c3magPlusPress}
\partial_t b + \D (u \otimes b - b \otimes u) = \frac{1}{2} \nabla (\pi_2 - \pi_1)
\end{equation}
so that $b$ solves the magnetic field equation if and only if $\nabla \pi_1 = \nabla \pi_2$. Of course, there is not much possibility for this equality not holding: by taking the divergence of \eqref{eq:c3magPlusPress}, we see that the difference $Q = \pi_2 - \pi_1$ is a harmonic polynomial
\begin{equation*}
\frac{1}{2} \Delta (\pi_2 - \pi_1) = \partial_j \partial_k (u_j b_k - u_k b_j) = 0,
\end{equation*}
but this does not imply that $Q = 0$. Let us give an example inspired by the uniform flow \eqref{ieq:c2Unif} presented above: consider 
\begin{equation}\label{eq:c3UnifEls}
\alpha(t, x) = f(t) \qquad \text{and} \qquad \beta(t, x) = -f(t),
\end{equation}
where $f \in C^\infty(\R ; \R^d)$ a smooth function. Then $(\alpha, \beta)$ solves the Elsässer system with pressure functions $\pi_1(t, x) = -f'(t) \cdot x = - \pi_2(t, x)$ but we have $\frac{1}{2} \nabla (\pi_2 - \pi_1) = f'(t) \neq 0$. We see that we have the same kind of issue here we had with the Euler system above: in other words, though a solution of the ideal MHD equations \eqref{ieq:c3mhd} always define a solution of the Elsässer system \eqref{eq:elsasser} through \eqref{eq:RChangeVariables}, the converse is false and a solution of \eqref{eq:elsasser} does not necessarily provide a solution of the ideal MHD equations \eqref{ieq:c3mhd} through the reverse transformations $(u, b) = \frac{1}{2}(\alpha + \beta, \alpha - \beta)$.

\medskip

Our question is the following: \textsl{on what condition does a bounded solution of \eqref{eq:elsasser} also define a solution of the ideal MHD system \eqref{ieq:c3mhd}?}

\begin{rmk}
The solution \eqref{eq:c3UnifEls} displayed above offers an interesting discussion as to whether it is physically relevant. Written in physical variables, it reads
\begin{equation}\label{eq:c3UnifEls2}
u(t, x) = 0 \qquad \text{and} \qquad b(t, x) = 2 f(t).
\end{equation}
On the one hand, it would seem that this solution has an obvious physical interpretation: a uniform magnetic field may be created in the interior of an infinite ideal cylindrical solenoid, so that \eqref{eq:c3UnifEls2} could be seen as generated by solenoidal-type currents ``at infinity''.

However, it turns out that this is an acceptable solution of the Maxwell equations (even in the non-relativistic regime suitable for MHD, see \cite{LebeLev}) only for non-variable fields $f'(t) = 0$. Because time variations of the magnetic field induce a nonzero electromotive force (by Faraday's law), the presence of an electric field $e$ only is possible if the fluid is in motion, as $e = - u \times b$.
\end{rmk}

\subsection{Weak Solutions of the MHD equations}

In this paragraph, we define the appropriate notions of weak solutions of systems \eqref{ieq:c3mhd} and \eqref{eq:elsasser}. We start by the usual MHD system \eqref{ieq:c3mhd} before doing so for the Elsässer system.

\begin{defi}\label{d:c3mhdWeak}
Let $T > 0$ and $u_0, b_0 \in L^\infty$ be two divergence-free functions. A couple $(u, b) \in L^2_{\rm loc} ([0, T[; L^\infty)$ is said to be a \textsl{bounded weak solution} of \eqref{ieq:c3mhd}\index{bounded weak solution!of the ideal MHD equations}, related to the initial data $(u_0, b_0)$ if there exists a (MHD) pressure $\pi \in \mc D'([0, T[ \times \R^d)$ such that
\begin{enumerate}[(i)]
\item the momentum equation is satisfied in the sense of distributions:
\begin{equation*}
\partial_t u + \D (u \otimes u - b \otimes b) + \nabla \pi = \delta_0(t) \otimes u_0(x) \qquad \text{in } \mc D'([0, T[ \times \R^d) \, ;
\end{equation*}
we refer to Definition \ref{d:c2weakEuler} for the meaning of the tensor product $\delta_0 \otimes u_0 \;$;
\item the magnetic field equation is satisfied in the sense of distributions: 
\begin{equation*}
\partial_t b + \D (u \otimes b - b \otimes u) = \delta_0(t) \otimes b_0(x) \qquad \text{in } \mc D'([0, T[ \times \R^d),
\end{equation*}
\item the divergence-free condition $\D (u) = 0$ holds in $\mc D' (]0, T[ \times \mathbb{R}^d)$.
\end{enumerate}
\end{defi}

\begin{defi}\label{d:c3mhdElsWeak}
Let $T > 0$ and $\al_0, \bt_0 \in L^\infty$ be two divergence-free functions. A couple $(\al, \bt) \in L^2_{\rm loc}([0, T[ ; L^\infty)$ is said to be a \textsl{bounded weak solution} of \eqref{eq:elsasser} related to the initial data $(\al_0, \bt_0)$ if there exist two pressure functions $\pi_1, \pi_2 \in \mc D'([0, T[ \times \R^d)$ such that
\begin{enumerate}[(i)]
\item the equation for $\al$ is satisfied in the sense of distributions:
\begin{equation*}
\partial_t \alpha + \D(\beta \otimes \alpha) + \nabla \pi_1 = \delta_0(t) \otimes \alpha_0(x) \qquad \text{in } \mc D'([0, T[ \times \R^d),
\end{equation*}
and likewise for $\bt$;
\item both divergence-free conditions $\D(\al) = \D(\bt) = 0$ hold in $\mc D' (]0, T[ \times \mathbb{R}^d)$.
\end{enumerate}
\end{defi}

\begin{rmk}
As above with the Euler system, we have the same ambiguity concerning initial data. If $(u, b)$ is as in Definition \ref{d:c3mhdWeak}, then there is no need for $u_0$ and $u(0)$ to be the same, even if the solution lies in $C^1_T(L^\infty)$. Likewise, there is no reason for $u$ to be continuous with respect to time. However, the magnetic field has a much more pleasant behavior: it automatically lies in $W^{1, 1}_T(W^{-1, \infty})$, and we have $b_0 = b(0)$. 

As for the Elsässer variables $(\al, \bt)$ in Definition \ref{d:c3mhdElsWeak}, they share the same problems the Euler solutions may have: both initial values $\alpha(0)$ and $\beta(0)$ may differ from the initial data $\alpha_0, \beta_0$, and both $\al$ and $\bt$ may be discontinuous with respect to time.
\end{rmk}

\begin{rmk}
As we have noted in \cite{CF3} (Section 4.1), to define weak solutions, it is only necessary, strictly speaking, that $u \otimes u - b \otimes b$ and $u \otimes b - b \otimes u$ be in $L^1_{\rm loc}(\R^d ; \R^d \otimes \R^d)$. Interestingly, this is absolutely equivalent to $\alpha \otimes \beta$ being locally integrable. But we will not need such a level of generality.
\end{rmk}

\subsection{Equivalence of the Systems}

The following statement is a sharp result for the equivalence of systems \eqref{ieq:c3mhd} and \eqref{eq:elsasser} in the framework of bounded solutions.

\begin{thm}\label{t:c3Els}
Let $T > 0$, $\al_0, \bt_0 \in L^\infty$ and $(\al, \bt) \in C^0 \big( [0, T[ ; L^\infty \big)$. Define $(u, b) = \frac{1}{2}(\alpha + \beta, \alpha - \beta)$, and $(u_0, b_0)$ accordingly. The assertions below are true.
\begin{enumerate}
\item Assume that $(u, b)$ is a weak solution of \eqref{ieq:c3mhd} that is related to the initial data $(u_0, b_0)$, in the sense of Definition \ref{d:c3mhdWeak}. Then $(\al, \bt)$ solves the Elsässer system \eqref{eq:elsasser} with initial data $(\al_0, \bt_0)$ and $\pi_1 = \pi_2 = \pi$.
\item Assume that $(\al, \bt)$ is a weak solution of \eqref{eq:elsasser} for the initial data $(\al_0, \bt_0)$. Then the following statements are equivalent:
\begin{enumerate}[(i)]
\item the functions $(u, b)$ solve the ``classical'' MHD system \eqref{ieq:c3mhd} with initial data $(u_0, b(0))$, and the difference $b(0) - b_0$ is a constant function;
\item we have, for all times $t \in [0, T[$, $b(t) - b(0) \in \mc S'_h$.
\end{enumerate}
In addition, if one of these equivalent conditions is fulfilled, then the pressure gradients $\nabla \pi_1$ and $\nabla \pi_2$ are equal.
\end{enumerate}
\end{thm}

\begin{proof}
Because the proof of this theorem bears strong similarities with that of Theorem \ref{t:c2Euler}, we only give an outline of the arguments.

\medskip

If $(u, b)$ is a weak solution of \eqref{ieq:c3mhd}, then adding and subtracting the weak forms of the momentum and magnetic field equations makes it clear that $(\al, \bt)$ solves the Elsässer system with initial data $(\al_0, \bt_0)$ with $\pi_1 = \pi_2$.

\medskip

On the other hand, assume that $(\al, \bt)$ is a solution of the Elsässer system \eqref{eq:elsasser}, according to Definition \ref{d:c3mhdElsWeak}, and with the initial data $(\al_0, \bt_0)$. As in the proof of Theorem \ref{t:c2Euler}, extend the solution $(\al, \bt, \pi_1, \pi_2)$ to $]- \infty, T[ \times \mathbb{R}^d$ and continue to note $(\al, \bt, \pi_1, \pi_2)$ the extension. By subtracting both equations for $\al$ and $\bt$, we see that the magnetic field satisfies 
\begin{equation*}
\partial_t b + \D (u \otimes b - b \otimes u) = \frac{1}{2} \nabla (\pi_2 - \pi_1) + \delta_0 (t) \otimes b_0(x).
\end{equation*}
As with $C^0_T(L^\infty)$ solutions of the Euler system, we apply the same mollification sequence $\big( K_\epsilon \big)_{\epsilon > 0}$ in order to deal with low regularity in the time variable. We have, by noting $(u_\epsilon, b_\epsilon) = K_\epsilon * (u, b)$ the convolution with respect to time,
\begin{equation*}
\partial_t b_\epsilon + K_\epsilon * \D (u \otimes b - b \otimes u) = \frac{1}{2} K_\epsilon * \nabla (\pi_2 - \pi_1) + K_\epsilon (t) b_0 (x) \qquad \text{in } ]- \infty, T - \epsilon[ \times \R^d.
\end{equation*}
By taking the divergence of this equation, we see that $K_\epsilon * \Delta (\pi_2 - \pi_1) = 0$, hence the functions $Q_\epsilon (t) = K_\epsilon * (\pi_2 - \pi_1) (t) \in \R[X]$ are harmonic polynomials. As in the proof of Theorem \ref{t:c2Euler}, we see that $\nabla Q_\epsilon = K_\epsilon * \nabla Q$ for some distribution $\nabla Q \in \mc D' (\R \times \R^d)$. We must find a condition under which we have $\nabla Q = 0$ in $]0, T[ \times \R^d$.

\medskip

We now notice that the quadratic term in the magnetic field equation lies in $\mc S'_h$, as shown by the first Bernstein inequality (Lemma \ref{l:c2bern}). For all times $t \in ]0, T[$,
\begin{equation*}
\left\| \chi (\lambda D) \D (u \otimes b - b \otimes u) \right\|_{L^\infty} = O \left( \frac{1}{\lambda} \right) \qquad \text{as } \lambda \rightarrow + \infty.
\end{equation*}
Therefore, by arguing with $b$ exactly as with $u$ at the end of the proof of Theorem \ref{t:c2Euler}, we get the equivalence of points \textit{(i)} and \textit{(ii)}.
\end{proof}

As with the Euler equations, well-posedness of a projected version of the ideal MHD system has been studied in the critical Besov space $B^1_{\infty, 1}$ (see the work of Miao and Yuan \cite{MY}). In the ideal MHD equations \eqref{ieq:c3mhd}, the pressure only appears in the momentum equation. This means that, in order to apply the Leray projection  we must require the velocity field $u$ to satisfy a special condition. The resulting system is
\begin{equation}\label{eq:mhdProjoU}
\begin{cases}
\partial_t u + \P \D (u \otimes u - b \otimes b) = 0\\
\partial_t b + \D (u \otimes b - b \otimes u) = 0,
\end{cases}
\end{equation}
The projected system \eqref{eq:mhdProjoU} is trivially equivalent to a projected system for the Elsässer variables, namely
\begin{equation}\label{eq:ElsProj}
\begin{cases}
\partial_t \al + \P \D (\bt \otimes \al) = 0\\
\partial_t \bt + \P \D (\al \otimes \bt) = 0.
\end{cases}
\end{equation}
However, solutions $(\al, \bt) \in C^0_T(L^\infty)$ of the Elsässer system, as in Definition \ref{d:c3mhdElsWeak}, solve these projected equations, with appropriate initial data, if and only if
\begin{equation*}
\big( \al(0), \bt(0) \big) = (\al_0, \bt_0) \qquad \text{and} \qquad \big( \al(t) - \al(0), \bt(t) - \bt(0) \big) \in \mc S'_h \text{ for } t \in [0, T[,
\end{equation*}
so that Theorem \ref{t:c3Els} remains a necessary step to recast the Elsässer system \eqref{ieq:c3mhd} into the classical MHD one. Relations between these different systems are summarized in the following diagram.
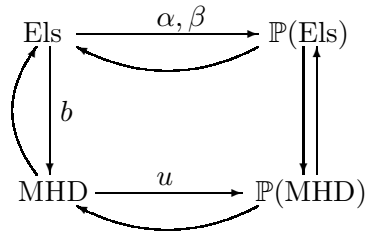
\begin{picture}(500, 90)
\put(150, 70){Els}
\put(170, 73){\vector(1, 0){68}}
\qbezier(170, 68)(204, 50)(238, 68)
\put(170, 68){\vector(-2, 1){0}}
\put(200, 76){$\al, \bt$}
\put(148, 10){MHD}
\put(160, 68){\vector(0, -1){48}}
\qbezier(155, 68)(137, 44)(155, 20)
\put(155, 68){\vector(1, 2){0}}
\put(164, 40){$b$}
\put(242, 70){$\P( {\rm Els})$}
\put(237, 10){$\P ({\rm MHD})$}
\put(177, 13){\vector(1, 0){55}}
\qbezier(170, 8)(204, -10)(238, 8)
\put(170, 8){\vector(-2, 1){0}}
\put(255, 68){\vector(0, -1){48}}
\put(200, 16){$u$}
\put(260, 20){\vector(0, 1){48}}
\end{picture}
In this diagram, Els stands for the Elsässer system \eqref{eq:elsasser} while MHD stands for the classical MHD system \eqref{ieq:c3mhd}. The names $\P({\rm Els})$ and $\P({\rm MHD})$ refer to, respectively, the projected equations \eqref{eq:ElsProj} and \eqref{eq:mhdProjoU}. The arrows are labeled with the conditions required to pass from a system to another, bare arrows need no conditions.

\begin{rmk}
In \cite{Kukavica} and \cite{KV}, it is shown\footnote{The authors of \cite{Kukavica} and \cite{KV} work with the Navier-Stokes equations, but minor adaptations allow to recover the same result in the context of the Euler equations, see Remark \ref{r:EulerGalilean} above} that any bounded solution of the Euler system \eqref{eq:c2Euler} can be obtained from a solution of the projected system \eqref{eq:c2EulerP} by means of a ``generalized Galilean transformation'' (see \eqref{ieq:c2GenGalInv} in the Introduction). We may wonder how much of this extends to MHD.

Firstly, ideal MHD, just as the Euler equations, is invariant under the operation of the Galileo group, provided the unknowns transform as
\begin{equation*}
u'(t,x) = u(t, x-Vt) + V, \quad b'(t, x) = b(t, x - Vt), \quad \pi'(t, x) = \pi(t, x - Vt).
\end{equation*}
This set of transformations is consistent with the non-relativistic nature of the MHD equations: they are deduced from the Euler system (which follows from Newtonian mechanics) and the \textsl{magnetic limit} of the Maxwell equations, that is the non-relativistic limit of the Maxwell equations in which the magnetic field dominates over the electric field $c|b| \gg |e|$ (contrary to the \textsl{electric limit} in which the converse holds). We refer to \cite{LebeLev}, \cite{Heras} and \cite{PFM} for more on this interesting topic. Our dissertation \cite{CobbPhD} also contains some comments on how these notions apply to MHD (pp. 28--31).

As a consequence, it is true that every solution $(u, b, \pi)$ of the Ideal MHD system \eqref{ieq:c3mhd}, labeled MHD in the diagram above, can be deduced from a solution $(U, B)$ of projected MHD system \eqref{eq:mhdProjoU} by a ``generalized Galilean transform''
\begin{equation*}
u(t, x) = U(t, x - F(t)) + f(t), \quad b(t, x) = B(t, x - F(t)), \quad \pi(t,x) = \Pi(t, x - F(t)) - f'(t) \cdot x,
\end{equation*}
where $F'(t) = f(t)$ and $\Pi = (- \Delta)^{-1} \partial_j \partial_k (U_j U_k - B_j B_k)$ is the pressure associated to $(U, B)$ as a solution of the projected system \eqref{eq:mhdProjoU}.

However, it is not obvious at all that there is a way to recover solutions $(\alpha, \beta, \pi_1, \pi_2)$ of the Elsässer system \eqref{eq:elsasser} in terms of solutions of the ideal MHD equations \eqref{ieq:c3mhd} and the operation of an appropriate invariance group.

\end{rmk}

%
%
%
%
%
%
%
%
%
%
%

\end{document}